\documentclass[a4paper,12pt]{article}
\usepackage{latexsym}
\usepackage{amsmath}
\usepackage{amssymb}
\usepackage{enumerate}

\usepackage{color}

\usepackage{theorem}
\newtheorem{theorem}{Theorem}[section]
\newtheorem{proposition}[theorem]{Proposition}
\newtheorem{lemma}[theorem]{Lemma}

\newtheorem{MTA}{Theorem A}
\newtheorem{MTB}{Theorem B}
\newtheorem{MTC}{Theorem C}

\theorembodyfont{\rmfamily}
\newtheorem{proof}{\textmd{\textit{Proof.}}}

\newtheorem{remark}[theorem]{Remark}
\newtheorem{example}[theorem]{Example}

\makeatletter

\@addtoreset{equation}{section}
\makeatother

\newcommand{\qedd}{\hfill \Box}

\title{The cut locus and distance function from a closed subset of a Finsler manifold
\footnote{
Mathematics Subject Classification (2010)\,:\,53C60, 53C22.}
}
\author{Minoru TANAKA $\cdot$ Sorin V. SABAU}
\date{}
\begin{document}

\maketitle

\begin{abstract}
We characterize the differentiable points of the distance function from a closed subset $N$ of an arbitrary dimensional Finsler manifold in terms of the number of $N$-segments. 
In the case of a 2-dimensional Finsler manifold, we prove the structure theorem of the cut locus of a closed subset $N$, namely that it
is a local tree, it is made of countably many rectifiable Jordan arcs 
except for the endpoints of the cut locus and that an intrinsic metric can be introduced in the cut locus and  its intrinsic and induced topologies coincide.
{\bf We should point out that these are new results even for Riemannian manifolds.}
\end{abstract}

\section{Introduction}
Among the variational problems, there is an interesting problem which is called the Zermelo navigation problem:\\
{\it Find the paths of shortest travel time from an originating point to
a destination under the influence of  a wind or a current when we travel by boat capable of a certain maximum speed.}\medskip\\  
 The shortest paths of this problem are geodesics of a Finsler  
metric. Notice that the shortest paths are geodesics of a Riemannian metric
only when there is no wind and no current. 

Hence, if we restrict ourselves to the variational problems of a Riemannian manifold, we must exclude such an interesting variational problem. This is the reason why we are interested in the variational problems of a Finsler manifold.

Any geodesic $\gamma$ emanating from a point $p$ in a compact Riemannian manifold looses the minimizing property at a point $q$ on $\gamma.$  Such a point $q$ is called a {\it cut point} of $p$ along $\gamma.$ The {\it cut locus} of a point $p$ is the set of all cut points along geodesics emanating from $p.$ 
The cut locus often appears as an obstacle when we try to prove some 
global structure theorems of a Riemannian manifold. For example, Ambrose's problem is easily solved if the cut locus of the base point is empty. Hebda (\cite{H}) and Itoh (\cite{I}) solved this problem affirmatively and independently in the 2-dimensional case, by proving that the 1-dimensional Hausdorff measure of the cut locus of a point in an arbitrary
 compact
subset is finite. It is still open for an arbitrary dimensional Riemannian manifold.  
The cut locus is also a vital notion in analysis, where the cut locus appears as a singular set. In fact, the cut locus of a point $p$ in a complete Riemannian manifold equals the closure of the set of all non-differentiable points of the distance function from the point $p.$

By being motivated by optimal control problems in space and quantum dynamics, 
the joint work (\cite{BCST})  was accomplished by Bonnard, Caillau, Sinclair and Tanaka. 
In this paper, the structure of the cut locus was determined for a class of 2-spheres of revolution which contains oblate ellipsoids, and this structure theorem gives global optimal results in orbital transfer and for Lindblad equations in quantum control.

The following property of the cut locus has played a crucial role in optimal transport problems (see 
\cite{V1}).
\medskip\\
{\it The distance function to the cut locus of a point $p$ of a complete Riemannian manifold is locally Lipschitz on the unit sphere in the tangent space at $p.$}\medskip\\
This property is often applied in many papers of optimal transport problems:
For example see  Loeper-Villani (\cite{LV}), Figalli-Rifford (\cite{FR}), Figalli-Rifford-Villani (\cite{FRV1}, \cite{FRV2}, \cite{FRV3}, \cite{FRV4}),  Figalli-Villani (\cite{FV}) and Villani (\cite{V2}).

Since H. Poincar\'e introduced the notion of the cut locus  in 1905,
the cut locus of a point or  a submanifold in a Riemannian manifold has been investigated by
many researchers 
( \cite{B}).
In spite of this fact, any structure theorem of the cut locus has not been established yet except for special Riemannian manifolds.
The main difficulty of formulating and proving such a theorem lies in 
the fact that the cut locus can be as complicated as a fractal set. 
In fact,
Gluck and Singer (\cite{GS}) constructed a smooth 2-sphere of revolution with positive Gaussian curvature admitting a point whose cut point is 
non-triangulable. Moreover, it is conjectured that the Hausdorff 
dimension of the cut locus would be a non-integer, if the manifold is not 
smooth enough.

However, it has been shown that the Hausdorff dimension of the cut locus of a point
is an integer  for a smooth Riemannian manifold (see \cite{ITd}) and 
 that the distance function to the cut locus of a
closed submanifold is locally Lipschitz for a smooth Riemannian manifold as well as for the Finslerian case
(see \cite{IT}, \cite{LN}).
 Hence, the cut locus has  enough differentiability for nonsmooth analysis (\cite{Cl}).

 Nevertheless, if we restrict ourselves to a surface, the structure theorem for the Riemannian cut locus has been established. Indeed, the detailed structure of the cut locus of  a 
 point or a smooth Jordan arc in a Riemannian 2-manifold have been thoroughly
 investigated (see \cite{SST}, \cite{H}). For example,
 Hebda proved in \cite{H} that the cut locus $C_p$ of a point $p$ in a 
 complete 2-dimensional 
  Riemannian manifold has a local-tree structure and that any two cut points
  of $p$ can be joined by a rectifiable arc in the cut locus $C_p$ if these two cut points are in the same connected component.
    Here, a topological space $X$ is called a {\it local tree }
     if for any $x\in X$ and any neighborhood $U$ of $x,$
      there exists an open neighborhood $V\subset U$ of $x$ such that any two points in $V$ can be joined by a unique continuous arc.

 In the present paper, the structure theorem of the cut locus of a 
       closed subset of a Finsler surface will be proved. 
       It should be noted that the investigations of the cut locus of a closed
       subset are scarce even in the case of a Riemannian manifold.
We will also investigate the differentiability of the distance function  from a closed subset of an arbitrary dimensional Finsler manifold. 
   

   It is well-known that the differentiability of the distance function is closely related to the cut locus.
      For example, it is known that the squared distance function from
       a point $p$ in a complete Riemannian manifold is differentiable at a point $q$ if and only if there exists a unique minimal geodesic segment joining
        $p$ to $q$ and that the squared distance function is smooth outside of the cut locus of $p.$ 
      One of our main theorems (Theorem A)  generalizes the facts above.

Let $N$ be a closed subset of a forward complete Finsler manifold $(M,F).$ Roughly speaking, the Finsler manifold is a differentiable manifold with a norm on each tangent space. The precise definition of  the Finsler manifold  and some necessary fundamental notation and  formulas will be reviewed later.

A (unit speed) geodesic segment $\gamma :[0,a]\to M$ is called an $N$-{\it segment} if
$d(N,\gamma(t))=t$ holds on $[0,a],$ where 
$d(N,p):=\min\{\; d(q,p)\; |\; q\in N\}$ for each point $p\in M.$
If a non-constant unit speed $N$-segment $\gamma:[0,a]\to M$ is maximal as an
 $N$-segment, then the endpoint $\gamma(a)$ is called a {\it cut point} of 
 $N$ along $\gamma.$
 The {\it cut locus} $C_N$ of $N$  is the set of all cut points along all non-constant $N$-segments.

One of our main theorems is on the distance function from a closed subset of a Finsler manifold.
The research
of the distance function $d_N(\cdot):=d(N, \cdot)$ from the closed subset 
$N$ is fundamental in the study of variational problems. 
For example, the viscosity solution of the Hamilton-Jacobi equation
is given by the flow of the gradient vector of the distance function $d_N,$ when $N$ is the smooth boundary of a relatively compact domain in Euclidean space (see \cite{LN}).

Although we do not assume any differentiability condition for the closed 
subset $N\subset M$, we may prove the following remarkable result.

\begin{MTA}
 Let $N$ be a closed subset of a forward complete arbitrary dimensional Finsler manifold $(M,F).$ Then, the distance function $d_N$ from the subset $N$ is differentiable at a point $q\in M\setminus N$ if and only if $q$ admits a unique $N$-segment.
\end{MTA}

Theorem B and Theorem C are our main theorems on the cut locus. 
 The theorems corresponding to Theorems B and C  have been  proved in \cite{ShT} for the cut locus of a compact subset of an Alexandrov surface.
We  should point out that the Toponogov comparison theorem was a key tool for proving main theorems in \cite{ShT}, but the Toponogov comparison theorem does not hold for Finsler manifolds. Hence, completely different proofs will be given to Theorems B and C.

\begin{MTB}
 Let $N$ be a closed subset of a forward complete 2-dimensional Finsler manifold $(M,F).$
Then, the cut locus $C_N$ of $N$ satisfies the following properties:
\begin{enumerate}
\item $C_N$ is a local tree and any  two cut points on the same connected component of $C_N$ can be joined by a rectifiable curve in $C_N.$
\item The topology of $C_N$ induced from the intrinsic metric $\delta$ coincides with the topology induced from $(M,F).$
\item  The space $C_N$ with the intrinsic metric $\delta$ is forward complete.
\item 
{The cut locus $C_N$ is a union of countably many Jordan arcs except for the endpoints of $C_N.$}
\end{enumerate}  
\end{MTB}

\begin{MTC}
 There exists a set ${\cal E}\subset [0,\sup d_N)$ of measure zero 
with the following properties:      
 \begin{enumerate}
\item  For each 
$t\in(0,\sup d_N)\setminus{\cal E},  $ the set  {\it $d_N^{-1}(t)$ consists of locally finitely many mutually disjoint arcs.
In particular, if $N$ is compact, then $d_N^{-1}(t)$ consists of finitely many mutually disjoint circles.}
\item
{\it  For each $t\in(0,\sup d_N)\setminus {\cal E},$ any point $q\in d_N^{-1}(t)$ admits at most two $N$-segments.}
\end{enumerate} 
\end{MTC}

\begin{remark}
Notice that the cut locus of a closed subset is not always closed (see Example \ref{Ex2.6}),
but the space $C_N$ with the intrinsic metric $\delta$ is forward complete for any closed subset of a forward complete Finsler surface. In the case where $N$ is a compact subset of an Alexandrov surface, 
all claims in Theorems B and C were proved  except for the third claim of Theorem B.
\end{remark}

\bigskip

Let us recall that  a {\it Finsler manifold} $(M,F)$ is an $n$-dimensional differential manifold $M$ endowed with a norm 
$F:TM\to [0,\infty)$ such that
\begin{enumerate}
\item $F$ is positive and differentiable on $\widetilde{TM}:=
TM\setminus\{0\}$;
\item $F$ is 1-positive homogeneous, i.e., $F(x,\lambda y)=\lambda F(x,y)$, $\lambda>0$, $(x,y)\in TM$;
\item the Hessian matrix $g_{ij}(x,y):=\dfrac{1}{2}\dfrac{\partial^2 F^2}{\partial y^i\partial y^j}$ is positive definite on $\widetilde{TM}.$
\end{enumerate}
Here $TM$ denotes the tangent bundle over the manifold $M.$
The Finsler structure is called {\it absolute homogeneous} if $F(x,-y)=F(x,y)$ because this leads to the homogeneity condition $F(x,\lambda y)=|\lambda| F(x,y)$, for any $\lambda\in \mathbb R$. 

By means of the Finsler fundamental function $F$ one defines the {\it indicatrix bundle} (or the Finslerian {\it unit sphere bundle}) by $SM:=\cup_{x\in M}S_xM$, where $S_xM:=\{y\in M\ | \ F(x,y)=1\}$. 

On a Finsler manifold $(M,F)$ one can  define the integral length of curves as follows. Let $\gamma:[a,b]\to M$ be a regular piecewise $C^{\infty}$-curve in $M$, and let $a:=t_0<t_1< \dots < t_k:=b$ be a partition of $[a,b]$ 
such that $\gamma|_{[t_{i-1},t_i]}$ is smooth for each interval $[t_{i-1},t_i]$, $i\in\{1,2,\dots,k\}$. The 
{\it integral length} of $\gamma$ is given by
\begin{equation}\label{integral length}
L(\gamma):=\sum_{i=1}^k\int_{t_{i-1}}^{t_i}F(\gamma(t),\dot\gamma(t))dt,
\end{equation}
where $\dot\gamma=\dfrac{d\gamma}{dt}$ is the tangent vector along the curve $\gamma|_{[t_{i-1},t_i]}$.
For such a partition, let us consider a regular piecewise $C^\infty$-map
\begin{equation}
\bar \gamma:(-\varepsilon,\varepsilon)\times[a,b]\to M,\quad (u,t)\mapsto \bar\gamma(u,t)
\end{equation}
such that $\bar\gamma|_{(-\varepsilon,\varepsilon)\times [t_{i-1},t_i]}$ is smooth for all $i\in\{1,2,\dots,k\}$, and
$\bar\gamma(0,t)=\gamma(t)$. Such a curve is called a regular piecewise $C^\infty$-{\it variation} of the base curve $\gamma(t)$, and the vector field  
$U(t):=\dfrac{\partial\bar\gamma}{\partial u}(0,t)$ is called the {\it variational vector field} of $\bar\gamma$. The integral length ${\cal L}(u)$ of $\bar\gamma(u,t)$ will be a function of $u$, defined as in \eqref{integral length}.

By a straightforward computation 
one obtains
\begin{equation}\label{arbitrary first variation}
\begin{split}
{\cal L}'(0)=& g_{\dot\gamma(b)}(\gamma,U)|_a^b+\sum_{i=1}^k\Bigl[g_{\dot\gamma(t_i^-)}(\dot\gamma(t_i^-),U(t_i))-
g_{\dot\gamma(t_i^+)}(\dot\gamma(t_i^+),U(t_i))\Bigl]\\
& -\int_a^bg_{\dot\gamma}(D_{\dot\gamma}{\dot\gamma},U)dt,
\end{split}
\end{equation}
where $D_{\dot\gamma}$ is the covariant derivative along $\gamma$ with respect to the Chern connection and $\gamma$ is arc length parametrized 
(see \cite{BCS}, p. 123, or \cite{S}, p. 77 for details of this computation as well as for the basis on Finslerian connections).  

A regular  piecewise $C^\infty$-curve $\gamma$ on a Finsler manifold is called a {\it geodesic} if ${\cal L}'(0)=0$ for all piecewise 
$C^\infty$-variations of $\gamma$ that keep its ends fixed.  
In terms of Chern connection a constant speed geodesic is characterized by the condition $D_{\dot\gamma}{\dot\gamma}=0$.

Let now $\gamma:[a,b]\to M$ be a unit speed geodesic and $\sigma:(-\varepsilon,\varepsilon)\to M$ a $C^\infty$-curve such that $\sigma(0)=\gamma(b)$. If one considers a $C^\infty$-variation 
$\bar\gamma:(-\varepsilon,\varepsilon)\times [a,b]\to M$ with one end fixed and another one on the curve $\sigma$, i.e. 
\begin{equation}
\bar\gamma(u,a)=\gamma(a),\qquad \bar\gamma(u,b)=\sigma(u),
\end{equation}
then formula \eqref{arbitrary first variation} implies that the integral length ${\cal L}(u)$ of the curve $\bar\gamma_u(t):=\bar\gamma(u,t)$, $t\in[a,b]$ satisfies the {\it first variation formula} (\cite{S}, p. 78):
\begin{equation}\label{first variation}
{\cal L}'(0)=g_{\dot\gamma(b)}({\dot\gamma(b)},\dot\sigma(0)).
\end{equation}
This formula is fundamental for our present study.

Using the integral length of a curve, one can define the Finslerian distance between two points on $M$. For any two points $p$, $q$ on $M$, let us denote by $\Omega_{p,q}$ the set of all piecewise $C^\infty$-curves $\gamma:[a,b]\to M$ such that $\gamma(a)=p$ and $\gamma(b)=q$. The map
\begin{equation}
d:M\times M\to [0,\infty),\qquad d(p,q):=\inf_{\gamma\in\Omega_{p,q}}L(\gamma)
\end{equation}
gives the {\it Finslerian distance} on $M$. It can be easily seen that $d$ is in general a quasi-distance, i.e., it has the properties
\begin{enumerate}
\item $d(p,q)\geq 0$, with equality if and only if $p=q$;
\item $d(p,q)\leq d(p,r)+d(r,q)$, with equality if and only if  $r$ lies on a
 minimal geodesic segment joining from $p$ to $q$ (triangle inequality).
\end{enumerate}

In the case where $(M,F)$ is absolutely homogeneous, the symmetry condition $d(p,q)=d(q,p)$ holds and therefore $(M,d)$ is a genuine metric space. We do not assume this symmetry condition in the present paper.

 Let us also recall that for a forward complete Finsler space $(M,F)$, the exponential map $\exp_p:T_pM\to M$ at an arbitrary point $p\in M$ is a surjective map (see \cite{BCS}, p. 152 for details). This will be always assumed in the present paper. 
 
 A unit speed geodesic on $M$ with initial conditions $\gamma(0)=p\in M$ and $\dot\gamma(0)=T\in S_pM$ can be written as $\gamma(t)=\exp_p(tT).$
Even though the exponential map is quite similar with the correspondent notion in Riemannian geometry, we point out two distinguished properties (see \cite{BCS}, p. 127 for  and details):
\begin{enumerate}
\item $\exp_x$ is only $C^1$ at the zero section of $TM$, i.e. for each fixed $x$, the map $\exp_xy$ is $C^1$ with respect to $y\in T_xM$, and $C^\infty$ away from it. Its derivative at the zero section is the identity map (Whitehead);
\item $\exp_x$ is $C^2$ at the zero section of $TM$ if and only if the Finsler structure is of Berwald type. In this case $\exp$ is actually $C^\infty$ on entire $TM$ (Akbar-Zadeh).
\end{enumerate}

\section{The distance function from a closed subset}\label{sec2}                                  
Let $N$ a closed subset of a forward complete Finsler manifold $(M,F).$ 
For each point $p\in M\setminus N,$
we denote by $\Gamma_N(p)$ the set of all unit speed $N$-segments to $p.$
Here a unit speed geodesic segment $\gamma :[0,a]\to M$ is called an $N$-{\it segment} if
$d(N,\gamma(t))=t$ holds on $[0,a],$ where 
$d(N,p):=\min\{\; d(q,p)\; |\; q\in N\}$ for each point $p\in M.$
If a non-constant unit speed $N$-segment $\gamma:[0,a]\to M$ is maximal as an $N$-segment, then the endpoint $\gamma(a)$ is called a {\it cut point} of $N$ along $\gamma.$ The {\it cut locus} $C_N$ is the set of all cut points along all non-constant $N$-segments.
Hence $C_N\cap N=\phi.$ Notice that there might exists a sequence of cut points convergent to $N$ if $N$ is not a submanifold. 
   
   \begin{remark}
  We discuss here only the forward complete case. Let us point out that in the Finsler case, unlikely the Riemannian counterpart, forward completeness is not equivalent to backward one, except the case when $M$ is compact.

   \end{remark}

 First, two versions of the first variation formula for the distance function from the closed set $N$ {will be} stated and proved. These formulas are fundamental for the study
 of the cut locus {hereafter}. In Proposition \ref{prop2.5}, as an application of the first variation formula, it is proved that the subset of the cut locus of $N$ that consists of all cut points of $N$ admitting at least two $N$-segments is dense in $C_N.$ 

The following proposition was proved in \cite{IT} for the distance function from a closed submanifold of a complete Riemannian manifold. Since the distance function on a Finsler manifold is not always symmetric, we have two versions of the first variation formula in our case.

\begin{proposition}{\bf (Generalized first variation formula, forward version)}\label{prop2.1}

Let $\{\gamma_i : [0,l_i]\to M\}$ be a convergent sequence of $N$-segments in an n-dimensional Finsler manifold $M$.
If the limit
\begin{equation}\label{limit formula forward}
v^f:=\lim_{i\to\infty}\frac{1}{F(\exp_x^{-1}(\gamma_i(l_i))}\exp_x^{-1}(\gamma_i(l_i)),
\end{equation} 
exists, then
\begin{equation}\label{eq2-2}
\begin{split}
g_{w_\infty}(v^f,w_\infty)=\min\{g_w(v^f,w)|\ & w \textrm{ is the unit velocity tangent vector}\\
&\textrm{ at } x 
 \textrm{ of } \gamma\in \Gamma_N(x)\}.
\end{split}
\end{equation}
Moreover,
\begin{equation}\label{eq2-3}
\lim_{i\to\infty}\frac{d(N,\gamma_i(l_i))-d(N,x)}{d(x,\gamma_i(l_i))}=g_{w_\infty}(v^f,w_\infty)
\end{equation}
holds.

Here $x:=\lim_{i\to \infty}\gamma_i(l_i),$  $w_\infty :=\lim_{i\to\infty}\dot\gamma_i(l_i)\in T_xM, $ and $\exp_x^{-1}$ denotes the local inverse map of the exponential map $\exp_x$ around the zero vector.
\end{proposition}

\begin{proof}
Let $\sigma_i: [0,d(x,\gamma_i(l_i) )]\to M$ denote the unit speed minimal geodesic segment emanating from $x$ to $\gamma_i(l_i)$, and 
hence
\begin{equation}\label{eq2-4}
\dot\sigma_i(0)=\frac{1}{F(\exp_x^{-1}(\gamma_i(l_i)))}\exp_x^{-1}(\gamma_i(l_i)).
\end{equation}
Let us choose  a positive constant $\delta$ in such a way that $\gamma(l-\delta)$ is a point of a strongly convex ball around $x.$ Here 
$\gamma:=\lim_{i\to\infty}\gamma_i$ and $l:= \lim_{i\to\infty}l_i.$ 

By the triangle inequality
\begin{equation*}
d(N,x)\leq d(N,\gamma_i(l-\delta))+d(\gamma_i(l-\delta),x),
\end{equation*}
and hence, we obtain
\begin{equation}{\label {eq2-5}}
d(N,\gamma_i(l_i))-d(N,x)\geq 
d(\gamma_i(l-\delta),\gamma_i(l_i))-d(\gamma_i(l-\delta),x).
\end{equation}

If we apply  the Taylor expansion formula for the function  $f(t):=d(\gamma_i(l-\delta), \sigma_i(t)), $ it follows from
the first variation formula \eqref{first variation} that there exists a positive constant $C$
such that  for any $i$ and any sufficiently small $|t|$
\begin{equation}\label{eq2-6}
d(\gamma_i(l-\delta),\sigma_i(t))\geq d(\gamma_i(l-\delta),x)+g_{w_i}(w_i,\dot\sigma_i(0))t-Ct^2,
\end{equation}
where $w_i$ denotes the unit velocity vector at $x$ of the minimal geodesic segment joining from $\gamma_i(l-\delta)$ to $x.$
 Thus, we obtain, by (\ref{eq2-5}) and (\ref{eq2-6})
\begin{equation}\label{eq2-7}
\liminf_{i\to\infty}\frac{d(N,\gamma_i(l_i))-d(N,x)}{d(x,\gamma_i(l_i))}\geq
\liminf_{i\to\infty}g_{w_i}(w_i,\dot\sigma_i(0)).
\end{equation}
Since $\lim_{i\to\infty}\gamma_i=\gamma,$
 we have
\begin{equation}\label{eq2-8}
\lim_{i\to \infty}w_i=w_\infty.
\end{equation}
From   \eqref{limit formula forward}, \eqref{eq2-4}, \eqref{eq2-7}, and  \eqref{eq2-8},
 it follows
\begin{equation}\label{eq2-9}
\liminf_{i\to\infty}\frac{d(N,\gamma_i(l_i))-d(N,x)}{d(x,\gamma_i(l_i))}\geq g_{w_{\infty}}(w_\infty,v^f).
\end{equation}

Let $\beta : [0,l] \to M$ be any unit speed $N$-segment in $\Gamma_N(x).$
The triangle inequality again gives 
\begin{equation}\label{eq2-10}
d(N,\gamma_i(l_i))\leq d(N,\beta(l-\delta))+d(\beta(l-\delta),\gamma_i(l_i)),
\end{equation}
where the positive constant $\delta$ is chosen in such a way that $\beta(l-\delta)$ lies in a strongly
convex ball  at $x$.

The relation \eqref{eq2-10} implies
\begin{equation} \label{eq2-11}
d(N,\gamma_i(l_i))-d(N,x)\leq d(\beta(l-\delta),\gamma_i(l_i))
-d(\beta(l-\delta),x)
\end{equation}
and from the Taylor expansion and the first variation formula \eqref{first variation} it results that   there exists a positive constant $C$ such that 
\begin{equation}\label{eq2-12}
d(\beta(l-\delta),\gamma_i(l_i))-d(\beta(l-\delta),x)\leq g_{w(\beta)}(w(\beta),\dot\sigma_i(0))d(x,\gamma_i(l_i))+Cd(x,\gamma_i(l_i))^2
\end{equation}
for any $i$, where $w(\beta):=\dot\beta(l)$.
Hence, 
for any $N$-segment $\beta\in\Gamma_N(x)$, we have
\begin{equation}\label{eq2-13}
\limsup_{i\to\infty}\frac{d(N,\gamma_i(l_i))-d(N,x)} 
{d(x,\gamma_i(l_i))}
\leq \lim_{i\to\infty} g_{w(\beta)}(w(\beta),\dot\sigma_i(0))=g_{w(\beta)}(w(\beta),v^f).
\end{equation}
In particular,
we obtain
\begin{equation}\label{eq2-14}
\limsup_{i\to\infty}\frac{d(N,\gamma_i(l_i))-d(N,x)}{d(x,\gamma_i(l_i))}\leq g_{w_\infty}(w_\infty,v^f).
\end{equation}
Now, the relation \eqref{eq2-3} follows from \eqref{eq2-9} and
\eqref{eq2-14}, while  (\ref{eq2-2}) is implied by \eqref{eq2-9} and \eqref{eq2-13}.
$\qedd$
\end{proof}

\begin{proposition}{\bf (Generalized first variation formula, backward version)}\label{prop2.2}

 Let $\{\gamma_i:[0,l_i]\to M\}$ be a convergent sequence of $N$-segments 
 in an n-dimensional Finsler manifold $M.$
If the limit
\begin{equation}\label{limit formula back}
v^b:=\lim_{i\to\infty}\frac{1}{F(\exp^{-1}_{\gamma_i(l_i)}(x))}\exp^{-1}_{\gamma_i(l_i)}(x)
\end{equation} 
exists,  
 then
\begin{equation*}
\begin{split}
g_{w_\infty}(-v^b,w_\infty)=\min\{g_w(-v^b,w)| \ & w \textrm{ is the unit velocity tangent vector }
\\
\textrm{ at } x 
&\textrm{ of } \gamma\in \Gamma_N(x)\}.
\end{split}
\end{equation*}

Moreover,
\begin{equation}\label{backward limit}
\lim_{i\to\infty}\frac{d(N,\gamma_i(l_i))-d(N,x)}{d(\gamma_i(l_i),x)}=g_{w_\infty}(-v^b,w_\infty)
\end{equation}
holds.

Here $x:=\lim_{i\to \infty}\gamma_i(l_i)$ and $w_\infty :=\lim_{i\to\infty}\dot \gamma_i(l_i)$.
\end{proposition}

\begin{proof} The proof is similar to that of Proposition \ref{prop2.1}, if we apply the Taylor expansion for the functions $d(\gamma_i(l-\delta),\sigma_i(t))$ and $d(\beta(l-\delta),\sigma_i(t)).$
 Here $\sigma_i: [-d(\gamma_i(l_i),x),0]\to M$ denotes the minimal geodesic segment emanating from $\gamma_i(l_i)$ to $x$.
$\qedd$ 
\end{proof}

The following theorem gives a necessary and sufficient condition for
the distance function from a closed subset to be differentiable at  a point.
 This theorem corresponds to the theorem of the differentiability of a Busemann function (see \cite{KTI}).

\begin{theorem}\label{th2.3}
Let $N$ be a closed subset of a forward complete n-dimensional Finsler manifold $M.$
Then, the distance function $d_N( \cdot):=d(N, \cdot)$ is differentiable at  a point  $q\in M\setminus N$
if and only if there exists a unique $N$-segment to $q.$
Furthermore,
the differential $(d d_N)_q$ of $d_N$ at a differentiable point $q\in M\setminus N$ satisfies that 
$$(d d_N)_q(v) =g_X(X,v)$$
for any $v\in T_q M.$ Here $X$ denotes the velocity vector at $q$ of the unique $N$-segment to $q.$
\end{theorem}
\begin{proof} 
Suppose that a point $q\in M\setminus N$ admits a unique $N$-segment $\alpha:[0,l]\to M.$
Let $v$ be any tangent vector with $F(v)=1.$
We obtain, by Proposition \ref{prop2.1}, 
$$\lim_{t\searrow 0}\frac{(d_N\circ\exp_q)(tv)-( d_N\circ\exp_q)(O_q)}{t}=g_{\dot\alpha(l)}(\dot\alpha(l),v).
$$
Hence, by Lemma \ref{lem2.4}, $d_N\circ\exp_q$ is differentiable at the zero vector $O_q.$
This implies that $d_N$ is differentiable at $q,$ since $(d\exp_q)_{O_q}$ is the identity map on the tangent space $T_q M$ at $q.$ 
Suppose next that $d_N$ is differentiable at a point $q\in M\setminus N$, and let $\alpha:[0,l]\to M$ be  a unit speed $N$-segment to  $q.$
It is clear that
\begin{equation*}
\lim_{t\nearrow l}\frac{d_N(\alpha(t))-d_N(q)}{t-l}=1.
\end{equation*}
Since $d_N(\alpha(t))$ is differentiable at $t=0,$
we obtain
\begin{equation}\label{eq2-17}
\lim_{t\searrow l}\frac{d_N(\alpha(t))-d_N(q)}{t-l}=1.
\end{equation}
Choose a decreasing sequence $\{t_i\}$ convergent to $l$ in such a way that
the sequence of $N$-segments to $\alpha(t_i)$ has a unique limit $N$-segment $\beta.$
Here the $N$-segment $\alpha$ is assumed to be extended as the geodesic on $[0,\infty).$
From Proposition \ref{prop2.1} and (\ref{eq2-17})  it follows that 
$$1=\lim_{i\to\infty}\frac{d_N(\alpha(t_i))-d_N(q)}{t_i-l}
=g_{\dot\beta(l)}(\dot\beta(l),\dot\alpha(l))$$
and 
$$g_{\dot\beta(l)}(\dot\beta(l),\dot\alpha(l))=\min\{g_{\dot\gamma(l)}(\dot\gamma(l),\dot\alpha(l)) \: | \:  \gamma\in \Gamma_N(q)\}.$$
Hence,
$g_{\dot\gamma(l)}(\dot\gamma(l),\dot\alpha(l))\geq 1,$ 
for any $\gamma\in \Gamma_N(q).$ 
Therefore, by Lemma 1.2.3 in \cite{S}, $\dot\gamma(l)=\dot\alpha(l)$ for any $\gamma\in\Gamma_N(q),$ and $q$ admits a unique $N$-segment.
$\qedd$
\end{proof}
\begin{lemma}\label{lem2.4}
Let $f: U\to R$ be a Lipschitz function on a  open convex subset around the zero vector $O$ of a Minkowski space $(V,F)$ with a Minkowski norm $F.$
Suppose that there exists a linear function $\omega : V\to R$ such that
for each ${e}\in F^{-1}(1),$ 
\begin{equation}\label{eq2-18}
\lim_{\lambda\searrow 0}\frac{f(\lambda  e) -f(O)-\omega(\lambda{ e})}{\lambda}=0. 
\end{equation}
Then,
\begin{equation}\label{eq2-19}
\lim_{F(v)\to 0}\frac{f(v)-f(O)-\omega(v)}{F(v)}=0, 
\end{equation}
i.e., $f$ is differentiable at the zero vector $O,$ and 
its differential at $O$ is $\omega.$
\end{lemma}

\begin{proof}
Choose any positive number $\epsilon$ and fix it.
Since $F^{-1}(1)$ is compact, we may choose finitely many elements $ e_1,\dots, e_k$ of $F^{-1}(1)$ in such a way that
for any $e\in F^{-1}(1),$ there exists some  $ e_i $ satisfying 
\begin{equation}\label{eq2-20}
F({ e-e_i})<\frac{\epsilon}{3L} \quad {\rm and}\quad |\omega({ e-e_i})|<\epsilon/3.
\end{equation}
 Here $L$ denotes a Lipschitz constant of the function $f.$
Let $v$ be any non-zero vector of $U$ and choose any $ e_i.$
By the triangle inequality,
\begin{equation}\label{eq2-21}
|f(v)-f(O)-\omega(v)|\leq |f(v)-f(\lambda { e_i})|+|f(\lambda e_i)-f(O)-\omega(\lambda e_i)|
+|\omega(v)-\omega(\lambda e_i )|,
\end{equation}
where $\lambda:=F(v).$
Since $f$ is a Lipschitz function with a Lipschitz constant $L,$
\begin{equation*}
|f(v)-f(\lambda e_i)| \leq  L F(v-\lambda {e_i})=\lambda  L F(e(v)-{e_i}),
\end{equation*}
where $ e(v):=\frac{v}{\lambda}\in F^{-1}(1).$
Hence, by means of (\ref{eq2-18})  and  (\ref{eq2-21}), we get
\begin{equation*}
\limsup_{F(v)\to 0}\frac{|f(v)-f(O)-\omega(v)|} {F(v)}\leq L\limsup_{\lambda\searrow 0}
| F( e(v)-{ e_i})  |+   \limsup_{\lambda\searrow 0}| \omega( e(v)-{e_i} ) |.
\end{equation*}
Here, by choosing $e_i$ so as to satisfy \eqref{eq2-20}, we get
$$\limsup_{F(v)\to 0}\frac{|f(v)-f(O)-\omega(v)|}{F(v)}\leq\epsilon.$$
Since $\epsilon$ is arbitrarily chosen, relation (\ref{eq2-19}) follows.

$\qedd$
\end{proof}                                                                                                                                                          

Let us recall that the cut locus $C_N$ is the set of all endpoints of non-constant maximal $N$-segments and that each element of $C_N$ is called a cut point of $N.$ 

The following proposition was proved by Bishop (\cite{Bh}) for the cut locus of a point in a Riemannian manifold. However, the proof of the following proposition is
direct, and hence 
 completely different from the one by Bishop.

\begin{proposition}\label{prop2.5}
Let $N$ be a closed subset of a forward complete $n$-dimensional Finsler manifold $M.$
Then the subset of $C_N${, which} consists of all cut points of $N$ admitting  (at least) two $N$-segments, is dense in the cut locus of $N$.
\end{proposition}
\begin{proof}
Let $p$ be  a cut point of $N$ {which} admits a unique $N$-segment. Suppose that there exists an open ball $B_{\delta_1}(p)$ each element of which admits a unique $N$-segment. From Theorem \ref{th2.3} {it follows that} the distance function $d_N|B_{\delta_1}(p)$ is a $C^1$-function and has no critical points. 
Hence, by the inverse function theorem, there exists a $C^1$-diffeomorphism $\varphi : (l-2\delta_2,l+2\delta_2)\times U_{\delta_3}\to \varphi( (l-2\delta_2,l+2\delta_2)\times U_{\delta_3} )\subset B_{\delta_1}(p)$
such that 
$\varphi(l,{O})=p,$ and $d(N,\varphi(t,q))=t$ on $(l-2\delta_2,l+2\delta_2)\times U_{\delta_3}.$
Here $l:=d(N,p),$ and $U_{\delta_3}$ denotes the open ball of radius $\delta_3$ centered at the origin $O$  in $R^{n-1}.$
For each point $x\in U_{\delta_3},$ let $\gamma_x : [0,l+\delta_2]\to M$ denote the $N$-segment passing through $\varphi(l,x).$ Since the set
$\bigcup_{ x\in U_\epsilon} \gamma_x[0,l+\delta_2]    $ is a neighborhood of $p$ for each $\epsilon\in(0,\delta_3),$  the  family $\{ \gamma_x|_{[0,l+\delta_2]}\} $ of  $N$-segments shrinks to
an  $N$-segment $\gamma: [0,l+\delta_2]\to M$ passing through $p$ when $\epsilon$ goes to zero.   This is  a contradiction. 

$\qedd$
\end{proof}

It is known that the cut locus of  a point in a complete Finsler manifold is closed (see for example \cite{BCS}). The following example shows that {in general} the cut locus of a closed subset in Euclidean plane is not closed.

\begin{example}\label{Ex2.6}
Choose any strictly decreasing sequence $\{\theta_n \}$ with $\theta_1\in(0,\pi)$ which is convergent to zero. Let $D$ denote  the closed  ball  with radius 1 centered at the origin of Euclidean plane $E^2$ {endowed with the standard Euclidean norm} and {let} $B_n$ {be} the open ball with radius 1 cantered at $q_n$, for each $n=1,2,3,\dots$
Here
 {$q_n \notin D$}
denotes the center of the circle with radius 1 passing through 
two points $(\cos\theta_n,\sin\theta_n)$ and  $(\cos\theta_{n+1},\sin\theta_{n+1}).$
 A closed subset $N$ of Euclidean plane is defined by
$$N:=D\setminus{\bigcup_{n=1}^\infty B_n}.$$
It is trivial {to see} that the sequence $\{q_n\}$ of cut points of $N$ converges to the point $(x,y)=(2,0).$
On the other hand, the point $(x,y)=(2,0)$ lies on the $N$-segment $\{(x,0)\; | \: 1\leq x\leq 3\}.$
This implies that the cut locus of the set $N$ is not  closed in $E^2.$

\end{example}

\section{The cut locus is a local tree}\label{sec3}                            
  From now on $N$ denotes a closed subset of a forward  complete 2-dimensional Finsler manifold $(M,F).$                                                                        
For each point $p\in M\setminus N,$
we denote by $\Gamma_N(p)$ the set of all unit speed $N$-segments to $p$,
 and by $B_{\delta}(q)$ the forward ball 
\begin{equation*}
B_{\delta}(q):=\{r\in M | d(q,r)<\delta   \},
\end{equation*}
centered at a point $q\in M$ and of radius $\delta.$

Let $x$ be a cut point of $N.$ Choose a small $\delta_0>0$ (to be fixed)
in such a way that  $B_{4\delta_0}(x)$ is a strongly convex neighborhood at $x$. For any 
$y\in C_N\cap B_{\delta_0}(x)$, each
connected component of
\begin{equation*}
B_{3\delta_0}(x)\setminus\{\gamma [0,d(N,y)]|\gamma\in\Gamma_N(y)\}
\end{equation*}
is called a {\it sector} at $y$. 

Choose any distinct two cut points  $y_0$ and $y_1$ of $N$ from $ B_{\delta_0}(x).$  
One can easily see that any                                                    
$\gamma\in \Gamma_N(y_0)$ does not pass through $y_1.$ Hence there exists a unique sector $\Sigma_{y_0}(y_1)$ at $y_0$ containing $y_1$. 
Let $\Sigma_{y_1}(y_0)$ denote the sector at $y_1$ containing $y_0$.
Since each $N$-segment to a point
in $B_{\delta_0}(x)$ intersects
$S_{2\delta_0}(x):=\{q\in M | d(x,q)=2\delta_0\}$
exactly once, 
the set
$$W(y_0,y_1):=\Sigma_{y_0}(y_1)\cap\Sigma_{y_1}(y_0)\cap B_{2\delta_0}(x)$$
is a 2-disc domain.
Furthermore, 
there exist exactly two open subarcs $I$ and $J$ of $S_{2\delta_0}(x)$ cut off by $N$-segments in $\Gamma_N(y_0)$ or  $\Gamma_N(y_1)$.
 If $\Gamma_N(y_0)$ or $\Gamma_N(y_1)$ consists of a single $N$-segment, then $I$ and $J$ have a common end point. 
Notice that for each point $r\in W(y_0,y_1)$, any 
$N$-segment to $r$ meets  $I$ or $J.$
Let $W_I(y_0,y_1)$ (respectively $W_J(y_0,y_1)$) denote the set of all points $r$ in $W(y_0,y_1)$ 
which admit an $N$-segment intersecting $I$ (respectively J).

\begin{lemma}\label{lem3.1}
Neither of $W_I(y_0,y_1)$ nor $W_J(y_0,y_1)$ is empty. Moreover,  if $y_0$ and $y_1$ are sufficiently close each other, then
 $ W_I(y_0,y_1)\cap W_J(y_0,y_1)$ is a subset of $B_{\delta_0}(x).$
\end{lemma}
\begin{proof}
Let $\gamma_I$ and $\gamma_J$ denote the $N$-segments in $\Gamma_N(y_0)$ that form part of the boundary of $W(y_0,y_1).$ Here we assume that $\gamma_I$ (respectively $\gamma_J$)
intersects $S_{2\delta_0}(x)$ at an end point of $I$ (respectively $J$).
Notice that $\gamma_I=\gamma_J$ holds if and only if  $\Gamma_N(y_0)$ consists of a single element. Take  $t_0\in(0,d(N,y_0))$ so as to satisfy that
$\gamma_I(t_0)$ and $\gamma_J(t_0)$ are points in $B_{\delta_0}(x).$
Choose strongly convex neighborhoods $B_\epsilon(\gamma_I(t_0))$($\subset B_{\delta_0}(x)$) and  $B_\epsilon(\gamma_J(t_0))$($\subset B_{\delta_0}(x)$) in such a way that
$B_\epsilon(\gamma_I(t_0))\cap B_\epsilon(\gamma_J(t_0))=\emptyset$
if $\gamma_I\ne\gamma_J.$
It is clear that 
$$D_I:=W(y_0,y_1)\cap B_\epsilon(\gamma_I(t_0))\quad{\rm
and}\quad
D_J:=W(y_0,y_1)\cap B_\epsilon(\gamma_J(t_0))$$
are disjoint if $\gamma_I\ne\gamma_J.$

In the case {when} $\gamma_I=\gamma_J,$
$D_I$ and $D_J$ denote the two connected components of 
 $B_\epsilon(\gamma_I(t_0))\setminus \gamma_I[0,d(N,y_0) )].$
In this case, we may assume that for each $t\in I$ sufficiently close to the intersection of $\gamma_I$
and $S_{ {2\delta}_0}(x),$ the minimal geodesic segment from $t$ to $\gamma_I(t_0)$ intersects $D_I,$ but does not intersect $D_J.$  

Suppose that {$\gamma_{I}\neq \gamma_{J}$ and} $W_I(y_0,y_1)$ or $W_J(y_0,y_1)$ is empty.
Without loss of generality, we may assume that $W_I(y_0,y_1)=\emptyset.$
Choose a sequence $\{q_n\}$ of points in $D_I$ converging to 
$\gamma_I(t_0).$
Let $\alpha$ be  a limit $N$-segment of the sequence $\{\alpha_n\},$ 
where $\alpha_n\in\Gamma_N(q_n).$ 
Since we {have} assumed that  $W_I(y_0,y_1)$ is empty, for each $n,$ $\alpha_n$ intersects $J.$ 
The $N$-segment $\alpha$ intersects the closure $\overline J$ of $J.$
Hence 
$\gamma_I(t_0)$ admits two $N$-segments $\alpha$ and $\gamma_I|_{[0,t_0]}$
if $\gamma_I\ne\gamma_J,$ that is, a contradiction. Therefore, {we must have} 
$\gamma_I=\gamma_J.$ 

Choose any point $q_J$ from $D_J,$ and fix it. Let 
$\alpha_J:[0,d(N,q_J)]\to M$ be an element of $\Gamma_N(q_J),$ and $\beta$ the unique minimal geodesic segment joining from $q_J$ to 
$\gamma_I(t_0)=\gamma_J(t_0).$ 
Since we assumed that  $W_I(y_0,y_1)$ is empty, $\alpha_J$ intersects $S_{2\delta_0}(x)$ at a point of $J.$
Then the three geodesic segments $\alpha_J,\;\beta$ and $\gamma_J|_{[0,t_0]}$ bound a 2-disc domain $D(\alpha_J,\beta)$ together with the subarc $c$ of $\overline J$ cut off by $\alpha_J$ and $\gamma_I=\gamma_J.$
Since we assumed that $W_I(y_0,y_1)$ is empty, the $N$-segment $\alpha_n$ intersects $J$
for each $n$ and the sequence $\{\alpha_n\}$ converges to $\gamma_I|_{[0,t_0]}.$
Therefore, for any sufficiently large $n,$ $\alpha_n$ intersects $c,$ the subarc of $\overline J.$ 
Hence
$\alpha_n$ passes through the disc domain $D(\alpha_J,\beta),$ and intersects $\beta$ at a point $p_n\in D_J.$
The subarc $\gamma_n$ of $\alpha_n$    with end points $p_n$ and $q_n$ is minimal and both end points are in $B_\epsilon(\gamma_I(t_0)).$ Since 
$B_\epsilon(\gamma_I(t_0))$ is a strongly convex ball, the subarc is entirely contained in the ball and joins $p_n\in D_J$ to $q_n\in D_I.$ Hence $\gamma_n$ meets $\gamma_I$ at a point in $B_\epsilon(\gamma_I(t_0)).$ This is {again} a contradiction, since both $\alpha_n$ and $\gamma_I$ are $N$-segments. The second claim can be proved by a similar argument {as} above.
$\qedd$
\end{proof}

\begin{lemma}\label{lem3.2}
For each $x\in C_N$ and each sector $\Sigma_x$ at $x$, there exists a sequence
of points in
$\Sigma_x\cap C_N$ convergent to $x.$
\end{lemma}
\begin{proof}
Suppose that there exists no cut point of $N$ in $B_\epsilon(x)\cap\Sigma_x$
for some sufficiently small positive $\epsilon.$ Let $\gamma$ denote an $N$-segment to a point in $\Sigma_x\cap S_{\epsilon/2}(x).$ Take any $\delta\in(0,\epsilon/2).$
For each point $y\in\Sigma_x\cap S_{\delta}(x),$ there exists an $N$-segment $\gamma_y$ to $y$. We get  a family of $N$-segments $\{\gamma_y\}_{y\in\Sigma_{x}\cap S_{\delta}(x)}.$ Since there exists no cut point of $N$ in $B_\epsilon(x)\cap\Sigma_x,$ the $N$-segment $\gamma$ is a restriction of $\gamma_y$ for some $y\in\Sigma_x\cap S_{\delta}(x).$ Since $\delta$ is chosen arbitrarily small, $\gamma$ is extensible to an $N$-segment to $x,$ which lies in the sector $\Sigma_x.$ This contradicts the definition of the sector.
$\qedd$
\end{proof}

\begin{lemma}\label{lem3.3}
$
W_I(y_0,y_1)\cap W_J(y_0,y_1)\neq \emptyset.$
\end{lemma}
\begin{proof}
It is clear that
the  set $W(y_0,y_1)$ is the union of $W_I(y_0,y_1)$ and $W_J(y_0,y_1),$ and that both $W_I(y_0,y_1)$ and $W_J(y_0,y_1)$ are relatively closed in $W(y_0,y_1)$. Hence $W_I(y_0,y_1)\cap W_J(y_0,y_1)\neq\emptyset$, since $W(y_0,y_1)$ is connected, and neither of $W_I(y_0,y_1)$ nor $W_J(y_0,y_1)$ is empty by Lemma \ref{lem3.1}.
$\qedd$
\end{proof}

Let us recall that an injective continuous map from the  open interval $(0,1)$  or closed interval $[0,1]$  of $\mathbb R$ into $M$ is called a {\it Jordan arc}. An injective continuous map from a circle $S^1$ into $M$ is called a {\it Jordan curve}. The image of a Jordan arc or a Jordan curve is also called a {\it Jordan arc} or {\it Jordan curve}, respectively.

\begin{lemma}\label{lem3.4}
Suppose that the cut points $y_0$ and $y_1$ of $N$ are sufficiently close each other, so that $ W_I(y_0,y_1)\cap W_J(y_0,y_1)\subset B_{\delta_0}(x).$ 
Then,
for each point $t\in I,$ there exists a unique point  $r\in W_I(y_0,y_1)\cap W_J(y_0,y_1)$ such that
there exists a sector at $r$ containing $t$ or there exists  an $N$-segment to $r$ which passes  through the point $t.$
\end{lemma}
\begin{proof}
Since $I$ and the closure $\overline I$ of $I$ are Jordan arcs,
we may assume that $I=(0,1)$ and $\overline I=[0,1].$
Suppose that there does not exist an $N$-segment to a point in $W_I(y_0,y_1)\cap W_J(y_0,y_1)$ passing through some $t\in(0,1)=I.$  Let $t_+\in I$ and $t_-\in I$ denote the minimum and
 the maximum of the following sets respectively:
$$I_+:=\bigcup_{r\in  W_I(y_0,y_1)\cap W_J(y_0,y_1)}    \{s\in[t,1] |\:{\rm there\: exists\: an\:element\: of\: }\Gamma_N(r){\rm \: passing \:through}\: s \}$$
$$I_-:=\bigcup_{r\in  W_I(y_0,y_1)\cap W_J(y_0,y_1)} \{s\in[0,t] |\:{\rm there\: exists\: an\:element\: of\: }\Gamma_N(r){\rm \: passing \:through}\: s \}$$  
It is clear that there exists a point $r_+$ (respectively $r_-$) in $W_I(y_0,y_1)\cap W_J(y_0,y_1)$ such that there exists an $N$-segment to $r_+$ (respectively $r_-$) passing
through $t_+$ (respectively $t_-$).
Suppose that $r_+\ne r_-.$ By applying Lemma \ref{lem3.3},
 we get a cut point $r\in W_I(r_-,r_+)\cap W_J(r_-,r_+)$    such that there exists an $N$-segment to $r$ passing through a point in $(t_-,t_+).$
Notice that $t_-<t<t_+,$ since we assumed that there does not exist an $N$-segment to a point in $W_I(y_0,y_1)\cap W_J(y_0,y_1)$ passing through the point $t.$
 This contradicts the definitions of $t_+$ and $t_-.$ Thus, $r_+=r_-,$ and  there exists a sector at $r_+=r_-\in W_I(y_0,y_1)\cap W_J(y_0,y_1)$ containing $t.$ 
The uniqueness of the existence of the point $r$ is clear, since $r\in W_I(y_0,y_1)\cap W_J(y_0,y_1)\subset C_N,$ and an $N$-segment does not intersect any other $N$-segment at its interior point.

$\qedd$
\end{proof}

\begin{proposition}\label{prop3.5}
Let $x$ be a cut point of $N$, and $B_{4\delta_0}(x)$ a strongly convex neighborhood at $x$. 
Then, there exists $\delta\in(0,\delta_0)$ such that any cut point 
$y\in B_{\delta}(x)\cap C_N$ can be joined {with} $x$ by a Jordan arc in $ B_{\delta_0}(x)\cap C_N.$ 
\end{proposition}
\begin{proof}
Choose a sufficiently small positive $\delta,$ so that $W_{I_z}(x,z)\cap W_{J_z}(x,z)
\subset B_{\delta_0}(x)$ for any $z\in B_{\delta}(x)\setminus\{x\}.$
Here $I_z$ and $J_z$ denote the 
open subarcs of $S_{2\delta_0}(x)$ that form part of the boundary of
$W(x,z):=\Sigma_x(z)\cap \Sigma_z(x),$ and  $\Sigma_z(x)$ (respectively $\Sigma_x(z)$) denotes the sector at $z$ (respectively at $x$) containing $x$ (respectively $z$). 
 Choose any  $y\in C_N\cap B_\delta(x)\setminus \{x\}$ and fix it.
Since $I$ and its closure $\bar{I}$ are  Jordan arcs, we may assume that $I=(0,1)$ and $\bar{I}=[0,1].$
Here $I$ and $J$ denote the subarc of $S_{2\delta_0}(x)$ corresponding to the cut point $y.$
 Here we assume that the  $N$-segment to $y$ (respectively $x$) forming the boundary of $W_I(x,y)\cap W_J(x,y)$ passing through the point $0$ (respectively $1$), which is an endpoint of $I.$

 We will construct a homeomorphism from $\bar{I}$ into $C_N\cap B_{\delta_0}(x)$. Choose any $t\in I$ and fix it. 
If there exists a cut point $z\in W_I(x,y)\cap W_J(x,y)$ such that a minimal geodesic segment in $\Gamma_N(z)$ passes through $t$, we define $\xi(t)=z$. 
Suppose that there is no such a cut point $z\in W_I(x,y)\cap W_J(x,y)$ for $t$. Then, from Lemma \ref{lem3.4}, it follows that there exists a sector $\Sigma_r$ at $r$
containing $t$ 
 for some cut point $r\in W_I(y_0,y_1)\cap W_J(y_0,y_1)$. 
We define $\xi(t)=r$ for such a $t$. Hence we have constructed a continuous map $\xi$ from $\bar{I}$ into $B_{\delta_0}(x)\cap C_N$, where we define  $\xi(0)=y$ and $\xi(1)=x$.  

It is clear that if $\xi(t_1)=\xi(t_2)$ holds for distinct $t_1,t_2\in \overline I$, then there exists an interval $[a,b]\subset\overline I$ such that $\xi |_{[a,b]}=\xi(t_1)$, $t_1,t_2\in [a,b]$. Hence there exist countably many mutually disjoint subintervals $\{I_n\}_n$ of $\overline I$, such that $\xi(t_1)=\xi(t_2)$ holds for distinct $t_1$, $t_2$ if and only if $t_1$ and $t_2$ are elements of a common $I_n$.

Let $f:[0,1]\to [0,1]$ be a continuous non-decreasing function such that $f(0)=0$, $f(1)=1$ and such that 
$f(t_1)=f(t_2)$ for distinct $t_1$, $t_2$ if and only if $t_1$ and $t_2$ lie in a common $I_n$ (the existence of the function $f$ is proved in Lemma 4.1.3 in \cite{SST}).

Then the curve $c:[0,1]\to B_{\delta_0}(x)\cap C_N$ defined by
\begin{equation*}
c(u):=\xi(\max f^{-1}(u))
\end{equation*}
is injective and continuous.
Hence, the cut points $y$ and $x$ can be joined by a Jordan arc in $B_{\delta_0}(x)\cap C_N$. 
$\qedd$
\end{proof}

A topological set $T$ is called a {\it tree} if 
 any two points in $T$ can be joined by a unique Jordan arc in $T$. 
Likely, a topological set $C$ is called a {\it local tree} if for every point
 $x\in C$ and for any neighborhood $U$ of $x$, 
 there exists a neighborhood $V\subset U$ of $x$ such that $C$ 
 is a tree. A point of a local tree $C$ is called  an {\it endpoint} of the local tree if there exists a unique sector at $x.$

\begin{theorem}\label{th3.6}
Let $N$ be a closed  subset of a (forward) complete 2-dimensional Finsler manifold $M.$
Then the cut locus of $N$ is a local tree.

\end{theorem}
\begin{proof}
Let $x$ be a cut point of $N,$ and $U$ a neighborhood of $x.$
Choose a strongly convex ball $B_{4\delta_0}(x)\subset U.$ Let $\delta$ be a positive number guaranteed in Proposition \ref{prop3.5}.
Let $\Sigma$ denote the intersection of all $\Sigma_y(x),$ where $y\in S_{\delta}(x)\cap C_N.$  From Proposition \ref{prop3.5}, it follows that any 
point $y\in\Sigma\cap C_N\cap B_{\delta}(x)$ can be joined by a Jordan arc 
$c$ in
 $B_{\delta_0}(x) \cap C_N.$ 
Since the curve $c$ does not intersect $S_{\delta}(x),$
 the curve lies in the  set $\Sigma\cap B_{\delta}(x)(\subset U).$
 Hence any  cut point
 of $N$ in $\Sigma \cap B_{\delta}(x)$ can be joined to the point $x$  by a Jordan arc in $\Sigma\cap C_N\cap B_{\delta}(x).$
 This implies that any two points in $\Sigma\cap C_N\cap B_{\delta}(x)$ can be joined by a
Jordan arc in $\Sigma\cap C_N\cap B_{\delta}(x)$ by way of $x.$
  Suppose that there exist two Jordan arcs in $\Sigma\cap C_N\cap B_{\delta}(x)$ joining two cut points of $N$ in $\Sigma\cap B_{\delta}(x).$ Then,   the Jordan arcs contain a Jordan curve $\alpha$ as a subset in the convex ball $B_{\delta}(x).$ Take a point $z$ in the domain bounded by $\alpha.$
 Any $N $-segments to $z$ intersect $\alpha\subset C_N.$ This is a 
 contradiction. Thus, any two points in  $\Sigma\cap C_N\cap B_{\delta}(x)$ is joined by a unique curve in the set. It is trivial that $\Sigma\cap C_N\cap B_{\delta}(x)$ is a neighborhood of $x$ since any $N$-segment to a point of 
 $ C_N\cap S_{\delta}(x)$ does not pass through the point $x.$ Therefore, $\Sigma\cap C_N\cap B_{\delta}(x)$ is a tree and a neighborhood of $x$ in $C_N.$  
$\qedd$
\end{proof}

\section{ Key lemmas}\label{sec4}

In this section, two key lemmas (Lemmas \ref{lem4.3} and \ref{lem4.4})
are proved. Before stating them, we need two fundamental lemmas which are true for any dimensional Finsler manifolds. 
The first one  is well known (see for example \cite{BCS}, Lemma 6.2.1).

\begin{lemma}\label{lem4.1}
Let $(M,F)$ be a (forward) complete Finsler manifold. Then,
for each positive number $a>0$, there exists a constant $\lambda(a)>1$ such that
\begin{equation*}
\lambda(a)^{-1}d(y,x)\leq d(x,y)\leq \lambda(a) d(y,x)
\end{equation*} 
for any $x,y\in B_a(p)$.
\end{lemma}

\begin{lemma}\label{lem4.2} 
Let $(M,F)$ be a (forward) complete Finsler manifold
and, 
 let $\alpha:[0,\infty)\times [0,2\pi]\to M$ denote the map defined by
\begin{equation*}
\alpha(t,\theta):=\exp_p(tv(\theta)),
\end{equation*}
where $v(\theta)$ denotes  a parametrization of the indicatrix curve $S_pM=\{v\in T_p M| \ F(p,v)=1\}$, and $\theta$ denotes the usual Euclidean angle.

Then for each $a>0$, there exists a positive constant $\mathcal C(a)$ such that
\begin{equation*}
F\left(\frac{\partial \alpha}{\partial \theta}(t,\theta)\right)\leq \mathcal C(a),
\end{equation*}
for any $t\in[0,a]$ and any $\theta\in[0,2\pi]$.
\end{lemma}

\begin{proof}
For each $\theta$,
\begin{equation*}
Y_\theta (t):=\frac{\partial\alpha}{\partial \theta}(t,\theta)
\end{equation*}
is a Jacobi field along the geodesic $\gamma_\theta (t):=\alpha(t,\theta)$ (see  p. 130 in \cite{BCS} or  p.167 in  \cite{S} for the details on the Jacobi equation in Finsler geometry).
The Jacobi field $Y_\theta (t)$ satisfies the differential equation
\begin{equation*}
D_TD_T Y_\theta (t)+R(Y_\theta (t),\dot\gamma_\theta(t))\dot\gamma_\theta(t)=0
\end{equation*}
with initial conditions 
\begin{equation*}
Y_\theta (0)=0,\quad D_TY_\theta(0)=\frac{\partial v}{\partial\theta}(0).
\end{equation*}
Here $D_T$ denotes the absolute derivative along $\gamma_\theta (t)$ with reference vector $T(t):=\dot\gamma_\theta(t)$ and $R$ denotes the 
$h$-curvature of $M$.
Since $Y_\theta (t)$ depends continuously on the initial conditions, there exists a constant $\mathcal C(a)$ such that 
\begin{equation*}
F(Y_\theta (t))\leq \mathcal C(a)
\end{equation*}
for any $t\in [0,a]$ and $\theta\in[0,2\pi]$. 
$\qedd$
\end{proof}

We define the length $l(c)$ {of} a continuous curve
 $c\; : \:[a,b]\to M$  by
\begin{equation}\label{d-length}
l(c):=\sup\{\sum_{i=1}^{k}\;d(c(t_{i-1}),c(t_i)) \: | \: a=:t_0<t_1<\dots<t_{k-1}<t_k:=b\}.
\end{equation}

From now on, we will fix a Jordan arc $c\; : \: [0,1]\to C_N$ in the cut locus of a closed subset
$N$ of a (forward) complete 2-dimensional Finsler manifold $(M,F).$

\begin{lemma}\label{lem4.3}
Let $[a,b]$ be a subinterval  of  $[0,1].$
Suppose that there exists a positive number $\epsilon_0$  such that
for each $t\in [a,b)$ (respectively $t\in(a,b]),$
$$\lim_{t\searrow t_0}D_N(c(t_0),c(t))>\epsilon_0$$
$$ ({\rm respectively} \quad \lim_{t\nearrow t_0}D_N(c(t),c(t_0))>\epsilon_0).$$
Then, the length $l(c)$ of $c$ is not greater than 
$\frac{1}{\epsilon_0}(d(N,c(b))-d(N,c(a)) ),$ i.e., $c$ is rectifiable.  Here, 
 $$D_N(x,y):=\frac{d(N,y)-d(N,x)}{d(x,y)}.$$
\end{lemma}

\begin{proof}
From our assumption, for 
 any sufficiently fine  subdivision 
$u_0:=a<u_1<\cdots<u_{n-1}<u_n:=b$
of $[a,b],$
$$d(N,c(u_{i+1}))-d(N,c(u_i))>\epsilon_0d(c(u_{i}),c(u_{i+1}))$$
holds for  each $i=0,1,\dots,n-1.$
Therefore, the length of $c$ is not greater than
\begin{equation*}
\frac{1}{\epsilon_0}\sum_{i=0}^{n-1} ( d(N,c(u_{i+1}))-d(N,c(u_i)) )=\frac{1}{\epsilon_0}(d(N,c(b))-d(N,c(a))).
\end{equation*}

$\qedd$

\end{proof}

\begin{lemma}\label{lem4.4}
Let $[a,b]$ be a subinterval of $[0,1].$
Suppose that there exists a point $p\in M$ such that
for each $t\in [a,b],$ the minimal geodesic segment from $p$ to 
$c(t)$ does not intersect $c[a,b]$ except $c(t)$ and such that
$c[a,b]$ is disjoint from the cut locus of $p$ and $P:=\{p\}.$ 
Suppose that there exists $\epsilon_0\in(0,1)$ such that 
\begin{equation}\label{eq4-1}
\lim_{t\searrow t_0}D_P(c(t_0),c(t))<\epsilon_0,
\end{equation}
 \begin{equation}\label{eq4-2}
({\rm respectively}\quad \lim_{t\nearrow t_0}D_P(c(t),c(t_0))<\epsilon_0)
\end{equation}
\begin{equation}\label{eq4-3}
\lim_{t\searrow t_0} D_P(c(t),c(t_0))<\epsilon_0
\end{equation}
\begin{equation}\label{eq4-4}
({\rm respectively}\quad \quad \lim_{t\nearrow t_0} D_P(c(t_0),c(t))<\epsilon_0)
\end{equation}
for each $t_0\in[a,b)$ (respectively $t_0\in(a,b].$)
Here,
$$D_P(x,y):=\frac{d(p,y)-d(p,x)}{d(x,y)}.$$
Then the curve $c$ is rectifiable on $[a,b].$

\end{lemma}
\begin{proof}

Let $v(\theta)$ denote a curve emanating from $v_0:=\frac{1}{F(\exp_p^{-1}(c(a)))} \exp_p^{-1}(c(a))$ in $S_p M.$ Here the parameter $\theta$ denotes the oriented  Euclidean angle measured from $v_0$ to $v(\theta).$ 
By the assumption of our lemma,
the curve $c$ is parametrized by $\theta$ ;
$$m(\theta)=\exp_p(\rho(\theta) v(\theta)), \quad \theta\in[0,\theta_0].$$
Here $\rho(\theta)=F(\exp_p^{-1}(c(t))),$  $v(\theta)=\frac{1}{\rho(\theta)}\exp_p^{-1}(c(t)),$ 
$m(0)=c(a),$ and $m(\theta_0)=c(b).$
From \eqref{eq4-1}, \eqref{eq4-2}, \eqref{eq4-3} and \eqref{eq4-4},
 it follows that for  any sufficiently fine subdivision 
$u_0:=0<u_1<u_2<\cdots<u_n:=\theta_0$
of $[0,\theta_0],$
\begin{equation}\label{eq4-5}
D_P(m(u_{i}),m(u_{i+1}) )   <\epsilon_0
\end{equation}
and 
\begin{equation}\label{eq4-6}
D_P(m(u_{i+1}),m(u_i))<\epsilon_0
\end{equation}
hold
for each $i=0,1,2,\cdots,n-1.$

Suppose first that 
$$l_i:=d(p,m(u_i))\leq l_{i+1}:=d(p,m(u_{i+1}))$$
for some fixed $i.$
By the triangle inequality,
\begin{equation}\label{eq4-7}
d(m(u_i),m(u_{i+1} )  )\leq d(m(u_i),\gamma_{i+1}(l_i))+l_{i+1}-l_i.
\end{equation}
Here $\gamma_{i+1} :[0,l_{i+1}]\to M$ denotes
the geodesic $\exp_p(t v(u_{i+1})).$ 

By applying Lemma \ref{lem4.2} to 
 the curve $\{\exp_p(l_i v(\theta))| u_i\leq \theta\leq u_{i+1}\},$
we get, by Lemma \ref{lem4.3},
\begin{equation}\label{eq4-8}
d(m(u_i),\gamma_{i+1}(l_i))\leq {\mathcal C}(\overline a)(u_{i+1}-u_i),
\end{equation}
where $\overline a :=\max\{d(p,c(t))\: | \: a\leq t\leq b\}.$
Combining  (\ref{eq4-5}), (\ref{eq4-7}) and (\ref{eq4-8}) we obtain
\begin{equation}\label{eq4-9}
d(m(u_i),m(u_{i+1}))\leq \frac{{\mathcal C}(\overline a)} {1-\epsilon_0} (u_{i+1}-u_i)
\end{equation}
for any $i$ with $d(p,m(u_i))\leq d(p,m(u_{i+1})).$

Suppose second that 
$$l_i=d(p,m(u_i))>l_{i+1}=d(p,m(u_{i+1}))$$ for some fixed $i.$
Then, by {a similar argument as} above,
we get
\begin{equation}\label{eq4-10}
d(m(u_{i+1}),m(u_i))\leq {\mathcal C}({\overline a} )(u_{i+1}-u_i)+
l_i-l_{i+1}.
\end{equation}
Combining (\ref{eq4-6}) and (\ref{eq4-10}), we obtain
$$ d(m(u_{i+1}),m(u_i)) \leq   \frac{ {\mathcal C}({\overline a})}   {1-\epsilon_0} (u_{i+1}-u_i).$$
Hence,
by Lemma \ref{lem4.1},
$$d(m(u_i),m(u_{i+1}))\leq \frac{\lambda(\overline a)}{1-\epsilon_0}{\mathcal C}(\overline a)(u_{i+1}-u_i)$$
for any $i$ with $d(p,m(u_i))>d(p,m(u_{i+1})).$
Therefore, the length $l(c)$ of $c$ does not exceed 
\begin{equation}\label{eq4-11}
\frac{\lambda(\overline a)}{1-\epsilon_0}{\mathcal C}(\overline a)
\theta_0,
\end{equation}
i.e.,   the curve $c$ is rectifiable. 
$\qedd$
\end{proof}

\section{Fundamental properties of
 a Jordan arc in the cut locus }\label{sec5}

Let us recall that $c:[0,1]\to C_N$ is a Jordan arc in the cut locus of the closed subset $N.$
For each $t\in [0,1)$ (respectively $t\in (0,1]$), let $\Sigma^+_{c(t)}$ (respectively $\Sigma^-_{c(t)}$) denote the sector at $c(t)$ that contains $c(t,t+\delta)$ (respectively $c(t-\delta,t)$) for some small $\delta>0$. 
Let $\alpha_t^+$ and $\beta_t^+$ (respectively $\alpha_t^-$ and $\beta_t^-$) 
denote the unit speed $N$-segments to $c(t)$ that form part of the boundary of $\Sigma^+_{c(t)}$ (respectively $\Sigma^-_{c(t)}$).
Notice that for each $t\in (0,1),$ $\alpha_t^+\ne\beta_t^+,$ and $\alpha_t^-\ne\beta_t^-.$

Then, with the notations above, we have the following important result.

\begin{proposition}\label{prop5.1}
Suppose that $\alpha_{t}^+\ne\beta_{t}^+$ for $t=0.$ Then
for each $t_0\in [0,1),$ 
 the following limits from the right exist:
\begin{equation}\label{eq5-1}
v^f(t_0)^+:=\lim_{t\searrow t_0}\frac{1}{F(\exp^{-1}_{c(t_0)}c(t))}\exp^{-1}_{c(t_0)}(c(t))
\end{equation}
and
\begin{equation}\label{eq5-2}
v^b(t_0)^+:=\lim_{t\searrow t_0}\frac{1}{F(\exp^{-1}_{c(t)}c(t_0))}\exp^{-1}_{c(t)}(c(t_0)).
\end{equation}
\end{proposition}

\begin{proof}
Since $S_{c(t_0)}M:=\{v\in T_{c(t_0)}M\; | \; F(v)=1\}$ is compact,
any sequence
$$\{ \frac{1} {F(\exp_{c(t_0)} ^{-1}c(t_0+\epsilon_i))}   \exp_{c(t_0)} ^{-1}  c(t_0+\epsilon_i)   \}$$ 
has a limit,
where
$\{\epsilon_i\}$ denotes a sequence of positive numbers convergent to zero. Let $v_0\in S_{c(t_0)}M$ be a limit of the sequence above. By choosing a subsequence, we may assume that the sequence has  a unique limit.
From Proposition \ref{prop2.1},
 it follows that 
$$
\lim_{i\to \infty}\frac{d(N,c(t_0+\epsilon_i))-d(N,c(t_0))}{d(c(t_0),c(t_0+\epsilon_i))}=g_{X}(X,v_0)=g_Y(Y,v_0),
$$
where $X$ and $Y$ are the tangent vectors of the unit speed $N$-segment $\alpha^+_{t_0}$ and $\beta^+_{t_0}$ at $c(t_0)$, respectively.
Since $X\ne Y,$
the space 
$\{Z\in T_{c(t_0)}M\; | \; g_X(X,Z)=g_Y(Y,Z)\}$
is a 1-dimensional linear subspace of $T_{c(t_0)}M.$
On the other hand,
it is clear that any limits of  
$\displaystyle\frac{1}{F(\exp_{c(t_0)}^{-1}c(t) )}  \exp_{c(t_0)}^{-1} c(t) $
as $t\searrow t_0$
lie in the common subarc $J^+(X,Y)$  of $S_{c(t_0)}M$ with endpoints $X,Y.$
Hence, the limit $v_0$ is the unique element
of $J^+(X,Y)\cap\{Z\in T_{c(t_0)}M\; | \; g_X(X,Z)=g_Y(Y,Z)\}. $
This implies
that the limit \eqref{eq5-1} exists.
By applying Proposition \ref{prop2.2}, we  can easily see that the limit \eqref{eq5-2} exists.
$\qedd$
\end{proof}

By reversing the parameter of $c$ in Proposition \ref{prop5.1}, we have the following proposition.
\begin{proposition}\label{prop5.2} 
Suppose that $\alpha_{t}^+\ne\beta_{t}^+$ for $t=1.$
Then for each $t_0\in (0,1],$  
 the following limits from the left exist:
\begin{equation*}
v^f(t_0)^-:=\lim_{t\nearrow t_0}\frac{1}{F(\exp^{-1}_{c(t_0)}c(t))}\exp^{-1}_{c(t_0)}(c(t))
\end{equation*} 
and
\begin{equation*}
v^b(t_0)^-:=\lim_{t\nearrow t_0}\frac{1}{F(\exp^{-1}_{c(t)}c(t_0))}\exp^{-1}_{c(t)}(c(t_0)).
\end{equation*} 
\end{proposition}

\begin{lemma}\label{lem5.3}
If   $\Gamma_N(c(0))$ consists of a unique  element $\alpha,$ then
$$\lim_{t\searrow 0}v^f(t)^+=\lim_{t\searrow 0}v^b(t)^-=\dot\alpha(l),
$$
where $l=d(N,c(0)).$
\end{lemma}
\begin{proof}
From the proof of Proposition \ref{prop5.1},
$v^f(t)^+$ (respectively $v^b(t)^-$)  is the unique element of the set
$J^+(X_t,Y_t)\cap \{Z\in S_{c(t)}M\; | \; g_{X_t}(X_t,Z)= g_{Y_t}(Y_t,Z)\}.$
Here $X_t=\dot\alpha_t^+(d(N,c(t)))$ (respectively $X_t=\dot\alpha_t^-(d(N,c(t)))$ ) and $Y_t=\dot\beta_t^+(d(N,c(t)))$(respectively $Y_t=\dot\beta_t^-(d(N,c(t)))$ ).
Since $\alpha$ is the unique element of $\Gamma_N(c(0)),$
the arc $J^+(X_t,Y_t)$ shrinks to $\dot\alpha(l)$ as $t\searrow 0.$
 Therefore,
$\lim_{t\searrow 0} v^f(t)^+=\lim_{t\searrow 0} v^b(t)^-=\dot\alpha(l).$

$\qedd$
\end{proof}
\begin{lemma}\label{lem5.4}
Suppose that $\alpha_{t}^+\ne\beta_{t}^+$ for $t=0.$ Then
 for each $t_0\in[0,1)$ 
$$
\lim_{t\searrow {t_0}}v^f(t)^+=v^f(t_0)^+,\quad  \textrm{and} \quad
\lim_{t\searrow{ t_0}}v^f(t)^-=\frac{-1}{F(-v^f(t_0)^+)} v^f(t_0)^+.$$
\end{lemma}
\begin{proof}
From the proof of Proposition  \ref{prop5.1},
 $v^f(t)^* ,$ where $*$ denotes $ +$ or $ -,$ is the unique element of the set
$$ J^*(X_t,Y_t)\cap \{Z\in T_{c(t)}M\: |\: g_{X_t}(X_t,Z)=g_{Y_t}(Y_t,Z)\}$$
for each $t\in(0,1).$
Here $X_t=\dot\alpha_t^+(d(N,c(t))),$ and $Y_t=\dot\beta_t^+(d(N,c(t))),$
 and $J^-(X_t,Y_t)$ denotes the complementary subarc of $J^+(X_t,Y_t)$ in $S_{c(t_0)}M.$
Since $ \lim_{t\searrow t_0}\alpha_t^*=\alpha_{t_0}^+$ and $\lim_{t\searrow t_0}\beta_t^*=\beta_{t_0}^+,$
we have
$\lim_{t\searrow t_0}X_t=X_{t_0},$ and
 $\lim_{t\searrow t_0}Y_t=Y_{t_0}.$
This implies that $\lim_{t\searrow _{t_0} }v^f(t)^+=v^f(t_0)^+$ and
$\lim_{t\searrow t_0}v^f(t)^-=    \frac{-1}{F(-v^f(t_0)^+)} v^f(t_0)^+.$
$\qedd$
\end{proof}

\begin{lemma}\label{lem5.5}
Suppose that $\alpha_{t}^+\ne\beta_{t}^+$ for $t=0.$ Then
for each $t_0\in [0,1),$  
 the following limits from the right exist:
\begin{equation}\label{eq5-3}
\lim_{t\searrow t_0}D_N(c(t_0),c(t))=g_{w(t_0)^+}(w(t_0)^+,v^f(t_0)^+)
\end{equation}
and
\begin{equation}\label{eq5-4}
\lim_{t\searrow t_0}D_N(c(t),c(t_0))=g_{w(t_0)^+}
(w(t_0)^+,v^b(t_0)^+),
\end{equation}
where 
$$
w(t_0)^+:=\dot \alpha^+_{t_0}(d(N,c(t_0))) \quad\textrm{ or }\quad
 \dot\beta^+_{t_0}(d(N,c(t_0))).
$$
Furthermore, for each compact subinterval $[a,b]\subset [0,1),$ there exists $\epsilon_0\in(0,1)$ such that
$$\lim_{t\searrow t_0}D_N(c(t_0),c(t))<\epsilon_0
\quad{\rm and}\quad                 
\lim_{t\searrow t_0}D_N(c(t),c(t_0))
<\epsilon_0$$
for each $t_0\in [a,b].$
\end{lemma}
\begin{proof}
From  Propositions \ref{prop2.1},  \ref{prop2.2}, and \ref{prop5.1}
it follows that for each $t_0\in[0,1)$ 
\begin{equation}\label{eq5-5}
N_+^f(t_0):=\lim_{t\searrow t_0}D_N(c(t_0),c(t))=g_X(X,v^f(t_0)^+)=g_Y(Y,v^f(t_0)^+)
\end{equation}
and
\begin{equation}\label{eq5-6}
N_+^b(t_0):=\lim_{t\searrow t_0}D_N(c(t),c(t_0))=g_X(X,v^b(t_0)^+)=g_Y(Y,v^b(t_0)^+)
\end{equation}
hold. 
Here $X:=\dot\alpha_{t_0}^+(d(N,c(t_0)))$ and $Y:=\dot\beta_{t_0}^+(d(N,c(t_0))).$
Hence \eqref{eq5-3} and \eqref{eq5-4} are clear.
We will prove the latter claim.
From Lemma 1.2.3 in \cite{S}, it is clear that 
$N_+^f(t_0)<1$ and $N_+^b(t_0)<1$ for each $t_0\in[0,1).$ Notice that 
$X\ne Y,$ since $\alpha_{t_0}^+\ne\beta_{t_0}^+.$
Suppose that 
 there exists a sequence $\{s_j\}$ of numbers in $ [a,b]$ satisfying
\begin{equation}\label{eq5-7}
\lim_{j\to \infty}N_+^f(s_j)=1\quad {\rm or}\quad \lim_{j\to\infty}N_+^b(s_j)=1
\end{equation}
Thus, by (\ref{eq5-5}), \eqref{eq5-6} and (\ref{eq5-7}),
\begin{equation}\label{eq5-8}
\lim_{j\to\infty} g_{X_j}(X_j,v^f(s_j)^+)=\lim_{j\to\infty} g_{Y_j}(Y_j,v^f(s_j)^+)=1
\end{equation}
or
\begin{equation}\label{eq5-9}
\lim_{j\to\infty} g_{X_j}(X_j,v^b(s_j)^+)=\lim_{j\to\infty} g_{Y_j}(Y_j,v^b(s_j)^+)=1
\end{equation}
holds.
Here, $X_j:=\dot\alpha_{s_j}^+(d(N,c(s_j)) )$ and $Y_j:=\dot\beta_{s_j}^+(d(N,c(s_j) )  ).$

By choosing a subsequence of $\{s_j\},$ we may assume that
$s_\infty:=\lim_{j\to\infty} s_j\in[a,b]$, 
$\alpha_\infty:=\lim_{j\to\infty}\alpha_j^+,$ $\beta_\infty:=\lim_{j\to\infty}\beta_j^+,$
$X_\infty:=\lim_{j\to\infty}X_j=\dot\alpha_{\infty}(d(N,c(s_\infty)),$ $Y_\infty:=\lim_{j\to\infty} Y_j=\dot\beta_{\infty}(d(N,c(s_\infty)),$ $\lim_{j\to\infty}v^f(s_j)^+$ and $\lim_{j\to\infty}v^b(s_j)^+(\in T_{c(s_\infty)} M)$ exist.
By (\ref{eq5-8}) and (\ref{eq5-9}), we obtain
\begin{equation}\label{eq5-10}
g_{X_\infty}(X_\infty,\lim_{j\to\infty}v^f(s_j)^+)
=g_{Y_\infty}(Y_\infty,\lim_{j\to\infty}v^f(s_j)^+)=1
\end{equation}
 or
\begin{equation}\label{eq5-11}
g_{X_\infty}(X_\infty,\lim_{j\to\infty}v^b(s_j)^+)
=g_{Y_\infty}(Y_\infty,\lim_{j\to\infty}v^b(s_j)^+)=1
\end{equation}
holds.

Since $\alpha_{\infty}$ and $\beta_{\infty}$ are $N$-segments
that form part of the sector $\Sigma_{c(s_{\infty})}^+$ or   $\Sigma_{c(s_{\infty})}^-$ at $c(s_{\infty}),s_{\infty}\in[a,b]\subset[0,1),$ 
it follows from our assumption that 
$X_{\infty}\ne Y_{\infty}.$ 
Thus, by Lemma 1.2.3 in \cite{S}, we get a contradiction from the equations
 \eqref{eq5-10} and \eqref{eq5-11}. This implies the existence of the number $\epsilon_0\in(0,1).$
$\qedd$
\end{proof}

Similarly, we have

\begin{lemma}\label{lem5.6}
Suppose that $\alpha_{t}^+\ne\beta_{t}^+$ for $t=0.$ Then
  for each $t_0\in (0,1),$   
the following limits from the left exist:
\begin{equation}\label{eq5-12}
\lim_{t\nearrow t_0}D_N(c(t_0),c(t))=g_{w(t_0)^-}(w(t_0)^-,v^f(t_0)^-)
\end{equation}
and
\begin{equation}\label{eq5-13}
\lim_{t\nearrow t_0}D_N(c(t),c(t_0))=g_{w(t_0)^-}(w(t_0)^-,v^b(t_0)^-),
\end{equation} 
where 
\begin{equation*}
w(t_0)^-:=\dot\alpha^-_{t_0}(d(N,c(t_0)) ) \textrm{ or } \dot\beta^-_{t_0}(d(N,c(t_0)) ).
\end{equation*}
Furthermore, for each compact subinterval $[a,b]\subset [0,1),$  there exists $\epsilon_0\in(0,1)$ such that
$$\lim_{t\nearrow t_0}D_N(c(t_0),c(t))<\epsilon_0, \quad {\rm and }\quad 
\lim_{t\nearrow t_0}D_N(c(t),c(t_0))<\epsilon_0$$
for each  $t_0\in (a,b].$

\end{lemma}

\section{Approximation by the distance function from a point}
\label{sec6}

\begin{lemma}\label{lem6.1}
Suppose that $\alpha_{t}^+\ne\beta_{t}^+$ for $t=0.$ Then
for any  $t_0\in[0,1)$ and each interior point $p$  of the $N$-segment
$\alpha_{t_0}^+,$ 
there exist positive numbers $\epsilon_1\in(0,1)$ and  $\delta_0$
such that 
$$\lim_{u\searrow t}D_P(c(t),c(u))<\epsilon_1\quad
    {\rm and}\quad
\lim_{u\searrow t}D_P(c(u),c(t))<\epsilon_1$$
for each  $t\in [t_0,t_0+\delta_0).$ Here $P:=\{p\}.$

\end{lemma}
\begin{proof} Suppose that  $t_0\in[0,1)$ is arbitrarily given. 
 Choose any interior point $p$
of the $N$-segment $\alpha_{t_0}^+.$
Since the point $c(t_0)$ is not a cut point of the point $p,$ 
there exists  $\delta_1\in(0,1-t_0)$ such that
the subarc $c[t_0,t_0+\delta_1]$ of $c$ is disjoint from the cut locus of $p.$   
Notice that the cut locus of $p$ is a closed subset of $M.$
By applying Lemma \ref{lem5.5} 
for the interval $[t_0,t_0+\delta_1],$
we get a number $\epsilon_0\in(0,1)$ satisfying 
\begin{equation}\label{eq6-1}
\lim_{u\searrow t} D_N(c(t),c(u))=g_{w(t)^+}(w(t)^+,v^f(t)^+)<\epsilon_0,\:
\end{equation}
and
\begin{equation}\label{eq6-2}
\lim_{u\searrow t} D_N(c(u),c(t))=g_{w(t)^+}(w(t)^+,v^b(t)^+)<\epsilon_0\:
\end{equation}
\
for each $t\in[t_0,t_0+\delta_1].$
Here $w(t)^+$ denotes  $\dot\alpha_t^+(d(N,c(t)))$
 in our argument.

For each $t\in[t_0,t_0+\delta_1],$ let
 $(\nabla d_p)_{c(t)}$ denote the (unit) velocity vector of  the minimal geodesic segment from $p$ to $c(t)$ at $c(t).$
Since
 $ (\nabla d_p)_{c(t_0)}=
w(t_0)^+$ 
and $\lim_{t\searrow t_0}w(t)^+=w(t_0)^+,$
we get a number  $\delta_0\in(0,\delta_1)$ so as to satisfy  that
$F((\nabla d_p)_{c(t)}-w(t)^+) )$ is sufficiently small for each $t\in[t_0,t_0+\delta_0],$ so that
\begin{equation}\label{eq6-3}
\left|  g_{(\nabla d_p)_{c(t)}}((\nabla d_p)_{c(t)},v^f(t)^+)-g_{w(t)^+}(w(t)^+,v^f(t)^+)\right|<\frac{1-\epsilon_0}{2}
\end{equation}
and
\begin{equation}\label{eq6-4}
\left|  g_{(\nabla d_p)_{c(t)}}((\nabla d_p)_{c(t)},v^b(t)^+)-g_{w(t)^+}(w(t)^+,v^b(t)^+)\right|<\frac{1-\epsilon_0}{2}
\end{equation}
hold for each $t\in[t_0,t_0+\delta_0].$
Therefore, by the triangle inequality and the equations \eqref{eq6-1}, \eqref{eq6-2}, \eqref{eq6-3} and \eqref{eq6-4},
 \begin{equation}\label{eq6-5}
 g_{(\nabla d_p)_{c(t)}  }( (\nabla d_p)_{c(t)},v^f(t)^+)<\frac{1+\epsilon_0}{2}\end{equation}
and
\begin{equation}\label{eq6-6}
 g_{(\nabla d_p)_{c(t)}  }( (\nabla d_p)_{c(t)},v^b(t)^+)<\frac{1+\epsilon_0}{2}\end{equation}
for each $t\in[t_0,t_0+\delta_0].$

On the other hand, by Propositions \ref{prop2.1}, \ref{prop2.2}  and \ref{prop5.1},
we obtain
\begin{equation}\label{eq6-7}
\lim_{u\searrow t}D_P(c(t),c(u))=g_{(\nabla d_p)_{c(t)}}((\nabla d_p)_{c(t)},v^f(t)^+)
\end{equation}

\begin{equation}\label{eq6-8}
\lim_{u\searrow t}D_P(c(u),c(t))=g_{(\nabla d_p)_{c(t)}}((\nabla d_p)_{c(t)},v^b(t)^+),
\end{equation}
for each $t\in [t_0,t_0+\delta_0].$
From  \eqref{eq6-5}, \eqref{eq6-6}, \eqref{eq6-7} and \eqref{eq6-8}, it is clear that
$\lim_{u\searrow t}D_P(c(t),c(u))$
 and $\lim_{u\searrow t}D_P(c(u),c(t))$ 
are less than $\epsilon_1:=\frac{1+\epsilon_0}{2}$ 
for each  $t\in [t_0,t_0+\delta_0].$

$\qedd$
\end{proof}

\begin{lemma}\label{lem6.2}
Suppose that $\alpha_{t}^+\ne\beta_{t}^+$ for $t=0.$ Then
for any  $t_0\in[0,1)$ and each interior point $p$  of the $N$-segment
$\alpha_{t_0}^+,$ 
there exist positive numbers $\epsilon_1\in(0,1)$ and  $\delta_0$
such that 
$$ \lim_{u\nearrow t}D_P(c(t),c(u))<\epsilon_1\quad
{\rm and}\quad \lim_{u\nearrow t}D_P(c(u),c(t))<\epsilon_1$$
for each  
  $t\in (t_0,t_0+\delta_0].$
\end{lemma}
\begin{proof} Suppose that  $t_0\in[0,1)$ is arbitrarily given. 
 Choose any interior point $p$
of the $N$-segment $\alpha_{t_0}^+.$ 
Since the point $c(t_0)$ is not a cut point of the point $p,$ 
there exists  $\delta_1\in(0,1-t_0)$ such that
the subarc $c[t_0,t_0+\delta_1]$ of $c$ is disjoint from the cut locus of $p.$ 
By applying Lemma  \ref{lem5.6}
for the interval $[t_0,t_0+\delta_1],$
we get a number $\epsilon_0\in(0,1)$ satisfying 
\begin{equation}\label{eq6-1p}
 \lim_{u\nearrow t} D_N(c(t),c(u))=g_{w(t)^-}(w(t)^-,v^f(t)^-)<\epsilon_0,
\end{equation}
and
\begin{equation}\label{eq6-2p}
 \lim_{u\nearrow t} D_N(c(u),c(t))=g_{w(t)^-}(w(t)^-,v^b(t)^-)<\epsilon_0
\end{equation}
for each  $t\in(t_0,t_0+\delta_1].$
Here  $w(t)^-$ denotes   $\dot\alpha_t^-(d(N,c(t)))$ in our argument.
Since
 $ (\nabla d_p)_{c(t_0)}=
w(t_0)^+$ 
and $\lim_{t\searrow t_0}w(t)^-=w(t_0)^+,$
we get a number  $\delta_0\in(0,\delta_1)$ so as to satisfy  that
$F((\nabla d_p)_{c(t)}-w(t)^-) )$ is sufficiently small for each $t\in[t_0,t_0+\delta_0],$ so that
\begin{equation}\label{eq6-3p}
\left|  g_{(\nabla d_p)_{c(t)}}((\nabla d_p)_{c(t)},v^f(t)^-)-g_{w(t)^-}(w(t)^-,v^f(t)^-)\right|<\frac{1-\epsilon_0}{2}
\end{equation}
and
\begin{equation}\label{eq6-4p}
\left|  g_{(\nabla d_p)_{c(t)}}((\nabla d_p)_{c(t)},v^b(t)^-)-g_{w(t)^-}(w(t)^-,v^b(t)^-)\right|<\frac{1-\epsilon_0}{2}
\end{equation}
hold for each $t\in(t_0,t_0+\delta_0].$
Therefore, by the triangle inequality and the equations 
   \eqref{eq6-1p}, \eqref{eq6-2p}, \eqref{eq6-3p} and \eqref{eq6-4p},
\begin{equation}\label{eq6-5p}
 g_{(\nabla d_p)_{c(t)}  }( (\nabla d_p)_{c(t)},v^f(t)^-)<\frac{1+\epsilon_0}{2}\end{equation}
and
\begin{equation}\label{eq6-6p}
  g_{(\nabla d_p)_{c(t)}  }( (\nabla d_p)_{c(t)},v^b(t)^-)<\frac{1+\epsilon_0}{2}
\end{equation}
for each $t\in(t_0,t_0+\delta_0].$
On the other hand, by Propositions \ref{prop2.1},  \ref{prop2.2} and  \ref{prop5.2},
we obtain
\begin{equation}\label{eq6-7p}
\lim_{u\nearrow t}D_P(c(t),c(u))=g_{(\nabla d_p)_{c(t)}}((\nabla d_p)_{c(t)},v^f(t)^-) 
\end{equation}
and 
\begin{equation}\label{eq6-8p}
\lim_{u\nearrow t}D_P(c(u),c(t))=g_{(\nabla d_p)_{c(t)}}((\nabla d_p)_{c(t)},v^b(t)^-)
\end{equation}
for each $t\in(t_0,t_0+\delta_0].$
From     \eqref{eq6-5p}, \eqref{eq6-6p}, \eqref{eq6-7p} and \eqref{eq6-8p},  it is clear that
$\lim_{u\nearrow t}D_P(c(t),c(u))$   and
$\lim_{u\nearrow t}D_P(c(u),c(t))$ 
are less than $\epsilon_1:=\frac{1+\epsilon_0}{2}$ 
for each  $t\in(t_0,t_0+\delta_0].$

$\qedd$
\end{proof}

\begin{lemma}\label{lem6.3}
If
$c(0)$ admits a unique $N$-segment, then there exist $\epsilon_0,\delta_0\in (0,1)$   satisfying  that 
$\lim_{u\searrow t} D_N(c(t),c(u))>\epsilon_0$ for each  $t\in [0,\delta_0),$
and $\lim_{u\nearrow t} D_N(c(u),c(t))>\epsilon_0$ for each  $t\in (0,\delta_0].$

 \end{lemma}

\begin{proof}

Since $c(0)$ admits a unique $N$-segment $\alpha : [0,l]\to M,$ 
$\lim_{t\searrow 0}w(t)^\pm=\dot\alpha(l).$
Hence, by Lemmas 5.3, we obtain
$$\lim_{t\searrow 0} g_{w(t)^+}(w(t)^+,v^f(t)^+)=1, \quad{\rm and }\quad
 \lim_{t\searrow 0}g_{w(t)^-}(w(t)^-,v^b(t)^-)=1.$$ 
Therefore, by  \eqref{eq5-3} and \eqref{eq5-13}, the existence of the numbers $\epsilon_0$ and $\delta_0$ is clear.

$\qedd$
\end{proof}

\begin{theorem}\label{th6.4}
Any   Jordan arc in the cut locus of a closed subset in a (forward) complete 2-dimensional Finsler manifold is  rectifiable.
\end{theorem}
\begin{proof}
Let $c: [0,1]\to C_N$ be a Jordan arc on the cut locus $C_N$ of a 
closed subset $N.$
Let $t_0\in[0,1)$ be arbitrarily given.
Suppose that $\alpha_t^+\ne\beta_t^+$ for $t=t_0.$ Choose any interior point $p$ of the $N$-segment $\alpha_{t_0}^+.$ Then, from Lemmas \ref{lem5.4}, \ref{lem6.1} and \ref{lem6.2},
it is clear that the point $p$ satisfies the hypothesis of Lemma \ref{lem4.4} for some interval $[t_0,t_0+\delta_0].$ Therefore, from Lemma \ref{lem4.4} it follows that 
$c|_{[t_0,t_0+\delta_0]  }$ is rectifiable.
Suppose next that $\alpha_t^+=\beta_t^+$ for $t=t_0.$
This means that $t_0=0$ and $c(0)$ admits a unique $N$-segment. From Lemmas \ref{lem4.3}
and \ref{lem6.3}, $c$ is rectifiable on $[0,\delta_2]=[t_0,t_0+\delta_2]$ for some positive $\delta_2.$
By reversing the parameter of $c$,
we have proved that for each $t_0\in[0,1],$ there exists an open interval containing
$t_0$ where $c$ is rectifiable. This implies that $c$ is rectifiable on $[0,1].$ 
$\qedd$
\end{proof}

\section{The topology induced by the intrinsic metric on the cut locus}\label{
sec7}

Let $N$ be a closed subset of a 2-dimensional (forward) complete Finsler manifold $M.$ 
By Theorems \ref{th3.6} and  \ref{th6.4}, any two cut points $y_1,y_2\in C_N$ can be joined by a rectifiable arc in the cut locus $C_N$ of $N$ 
if $y_1$ and $y_2$ are in the same connected component of the cut locus 
of the closed subset $N.$
Therefore, we can define the {\it intrinsic metric} $\delta$ on $C_N$ as follows:
\begin{enumerate}
\item if $y_1,y_2\in C_N$ are in the same connected component,
\begin{equation*}
\delta(y_1,y_2):=\inf\{l(c)|\ c\ \textrm{is a rectifiable arc in }C_N\ \textrm{joining } y_1\ \textrm{and } y_2\},
\end{equation*} 
\item otherwise $\delta(y_1,y_2):=+\infty$.
\end{enumerate}
Recall that  $l(c)$, {given in \eqref{d-length}},
denotes the length for a continuous curve $c\; : \:[a,b]\to M$. 
It is fundamental that $l(c)$ equals the integral length $L(c)$ defined in \eqref{integral length} for any piecewise $C^1$-curve $c$ (see  \cite{BM} for this proof). We will prove in Theorem \ref{th7.5} that $l(c)=L(c)$ holds for any Lipschitz 
 continuous
  curve $c.$ Notice that for any Lipschitz curve $c,$
  $c(t)$ is differentiable for almost all $t.$
 
The following theorem for a Lipschitz function is a key tool in the argument below.
The proof is {immediate by taking into account Theorem 7.29 in \cite{WZ}, for example}.
\begin{theorem}\label{th7.1}
If $f:[a,b]\to R$ is a Lipschitz function, its derivative function $f '(t)$ exists for almost all $t$ and 
$$f(t)=f(a)+\int_a^t f'(t) dt$$
holds for any $t\in[a,b].$ 
\end{theorem}

\begin{lemma}\label{lem7.2}
 Let $\gamma:[0,1]\to M$ be a Lipschitz curve on $M$ and let $f:[0,1]\to [0,\infty)$ be the distance function $f(t):=d(q,\gamma(t))$ from a point $q.$   Then
 $f$ is a  Lipschitz function and 
 $f'(t)\leq F(\dot\gamma (t))$ holds for almost all $t\in (0,1)$. 

 \end{lemma}
 \begin{proof}

 It follows from the triangle inequality that the function  $f$ is Lipschitz.
 
From Theorem \ref{th7.1} it follows  that the curve $\gamma$ and
 function $f$ are differentiable almost everywhere.  Using again the  triangle inequality, for any small  positive number $h,$ we have
 \begin{equation*}
 f(t+h)= d(q,\gamma(t+h))\leq f(t)+d(\gamma(t),\gamma(t+h))
 \end{equation*}
 and therefore, if $f'(t)$ exists, then
 \begin{equation*}
 f'(t)=\lim_{h\searrow 0}\frac{f(t+h)-f(t)}{h}\leq  \lim_{h\searrow 0}\frac{d(\gamma(t),\gamma(t+h))}{h}=F(\dot\gamma(t)),
 \end{equation*}
 where we have used the well-known Busemann-Mayer formula (see the original paper \cite{BM}, or a more modern treatment \cite{BCS}, p. 161). See also the proof of Lemma \ref{lem7.6}.

 $\qedd$
 \end{proof}
 
\begin{lemma}\label{lem7.3}
The length $l(\gamma)$ of the Lipschitz curve $\gamma$ satisfies
\begin{equation*}
l(\gamma |_{[0,a]})+l(\gamma |_{[a,a+h]})=l(\gamma |_{[0,a+h]}),
\end{equation*} 
for any  non-negative $ h\leq1-a$. In particular, the function $l(c|_{[0,t]}) $ is Lipschitz.
\end{lemma} 
 \begin{proof}
 It follows directly from  (\ref{d-length}).
  $\qedd$
 \end{proof}

\begin{lemma}\label{lem7.4}
 For almost all $t\in(0,1)$ we have
 \begin{equation*}
 \frac{d}{dt}l(\gamma |_{[0,t]})\geq F(\dot\gamma (t)).
 \end{equation*}
 \end{lemma}
 \begin{proof}
Suppose that the function $l(\gamma|_{[0,t]})$  and $\gamma$ are  differentiable at $t=t_0.$
 Using Lemma \ref{lem7.3}, we have
 \begin{equation*}
 \left(\frac{d}{dt}\right)_{t_0} l(\gamma |_{[0,t]})=\lim_{h\searrow 0}\frac{l(\gamma |_{[0,t_0+h]})-l(\gamma |_{[0,t_0]})}{h}=\lim_{h\searrow 0}\frac{l(\gamma |_{[t_0,t_0+h]})}{h},
 \end{equation*}
 and from  (\ref{d-length}) it follows that 
 \begin{equation*}
 \left(\frac{d}{dt}\right)_{t_0} l(\gamma |_{[0,t_0]})\geq \lim_{h\searrow 0}\frac{d(\gamma(t_0),\gamma(t_0+h))}{h}
 =F(\dot\gamma (t_0)).
 \end{equation*}
 $\qedd$
 \end{proof}

 We can formulate now one of the main results of this section.
 
 \begin{theorem}\label{th7.5}
 For any Lipschitz curve $\gamma:[0,1]\to M$,  $l(\gamma)$ equals $L(\gamma).$ 
 \end{theorem}
 \begin{proof}
 For any  subdivision $t_0=0<t_1<t_2<\dots <t_n=1$ of $[0,1]$, from Theorem \ref{th7.1} and  Lemma \ref{lem7.2}, it follows that 
 \begin{equation*}
 d(c(t_i),c(t_{i+1}))=\int_{t_i}^{t_{i+1}}\frac{d}{dt}d(c(t_i),c(t))dt\leq \int_{t_i}^{t_{i+1}} F(\dot\gamma (t))dt.
 \end{equation*}
By summing, it follows that 
 \begin{equation}\label{eq7-2}
 \sum_{i=0}^{n-1} d(c(t_i),c(t_{i+1}))\leq \int_0^1F(\dot\gamma (t))dt=L(\gamma).
 \end{equation} 
  On the other hand, by Theorem \ref{th7.1} and Lemma \ref{lem7.4}, 
  \begin{equation}\label{eq7-3}
  l(\gamma)=\int_0^1\frac{d}{dt}l(\gamma |_{[0,t]})dt\geq \int_0^1 F(\dot\gamma(t))dt=L(\gamma).
  \end{equation}
 The conclusion follows from the relations \eqref{eq7-2} and \eqref{eq7-3}.
 $\qedd$
 \end{proof}

It can be seen that the function $\delta$ is a quasi-distance on $C_N$. 
It is clear that  Lemma \ref{lem4.1} holds for the quasi-distance function  $\delta.$ Thus $\lim_{n\to \infty}\delta(x,x_n)=0$ is equivalent to $\lim_{n\to \infty}\delta(x_n,x)=0$ for any sequence $\{x_n\}.$ 
In the case where  $F$ is absolute homogeneous, $\delta$ is a genuine distance function on $C_N$.

Let $c:[0,a]\to C_N$ be a Jordan arc parametrized by arclength, i.e., 
$l(c|_{[0,t]})=t$ for all $t\in[0,a]$, {where $l$ is given in \eqref{d-length}}.
By definition {we have}
\begin{equation}\label{eq7-4}
d(c(t_1),c(t_2))\leq \delta(c(t_1),c(t_2))
\leq l(c|_{[t_1,t_2]})=|t_1-t_2|
\end{equation}
for any $t_1,t_2\in[0,a].$
This implies that $c:[0,a]\to M$ is a Lipschitz map (with respect to $d$) and 
hence $c$ is differentiable for almost all $t.$ We will prove that $c$ is a unit speed curve, i.e., $F(\dot c(t))=1$ for almost all $t.$

\begin{lemma}\label{lem7.6}
For almost all $t,$
$F(\dot c(t))=1.$ Conversely, if $F(\dot c(t))=1$ for almost all $t,$
then $c$ is parametrized by arclength.
\end{lemma} 
\begin{proof}
Suppose that $c$ is differentiable at $t_0\in(0,a).$
Since 
$$\lim_{t\searrow t_0}\frac{d(c(t_0),c(t))}{t-t_0}=
\lim_{t\searrow t_0}F\left(\frac{\exp^{-1}_{c(t_0)}c(t)}{t-t_0}\right)=
F((d\exp_{c(t_0)}^{-1})_{O_{c(t_0)}} \dot c(t_0))=F(\dot c(t_0)),$$
we get, by \eqref{eq7-4},
$F(\dot c(t_0))\leq 1.$ Hence, $F(\dot c(t))\leq 1$ for almost all $t.$
Since $a=l(c)=\int_0^a F(\dot c(t))dt$, by Theorem \ref{th7.5}, {it results}
$\int_0^a \left(1-F(\dot c(t))\right)dt=0.$
Thus, $F(\dot c(t))=1$ for almost all $t,$ since $1-F(\dot c(t))\geq 0.$
$\qedd$
\end{proof}

The following two lemmas {follow immediately} from Propositions \ref{prop2.1}, \ref{prop2.2} and
 Lemma \ref{lem7.6}.
 \begin{lemma}\label{lem7.7}
 For almost all $t$, {we have}
 $$\dot c(t)=v^f(t)^+=v^b(t)^-.$$
 \end{lemma}

 \begin{lemma}\label{lem7.8}
 Suppose that $c(t)$ is differentiable at $t=t_0.$
 If $d_N\circ c(t)$ is differentiable at $t=t_0,$ then
 $(d_N\circ c)'(t_0)=g_X(X,\dot c(t_0)),$ where $X$ denotes the velocity vector of an $N$-segment to $c(t_0).$
 
  \end{lemma}

 \begin{remark}\label{rem7.9}
 It is clear that $(d_N\circ c)(t)$ is differentiable at $t=t_0$ if $c(t_0)$ is not a branch cut point and if $c(t)$ is differentiable at $t=t_0.$ Here a cut point $c(t)$ is called a {\it branch cut point} if $c(t)$ admits more than two 
 sectors. It will be proved in Lemma \ref{lem8.1} that there exist at most countably many branch cut points.  
 \end{remark}

\begin{lemma}\label{lem7.10}
If $\{c(t_n)\},$ where $t_n\in(0,a],$ is a sequence of points on the curve $c$ convergent to $c(0)$ (with respect to $d$), then 
$\lim_{n\to\infty}\delta(c(0),c(t_n))=0.$
\end{lemma}
\begin{proof}
Let $\{ t_{n_i}\}$ be any convergent subsequence of $\{t_n\}.$
Since $\lim_{n\to\infty} d(c(0),c(t_n))=0$ and $c$ is continuous,
we get $d(c(0),c(t_\infty))=0,$ where $t_\infty$ denotes the limit of $\{ t_{n_i}\}.$
Thus, $c(0)=c(t_\infty)$ and $t_\infty=0.$ This implies that $\lim_{n\to\infty}t_n=0.$
By definition,
$\delta(x,c(t_n))\leq l(c|_{[0,t_n]})=t_n.$ Therefore,
$\lim_{n\to\infty}\delta(c(0),c(t_n))=0.$

$\qedd$
\end{proof}

\begin{lemma}\label{lem7.11}
Let   $\{x_n\}$  be a sequence of cut points of $N$ convergent to 
a cut point  $x.$
If all $x_n$
lie in a  common sector $\Sigma_x$ at $x,$
then
$\lim_{n\to\infty} \delta(x,x_n)=0.$

\end{lemma}
\begin{proof}
For each $n,$ let $e_n:[0,a_n]\to C_N$ denote a unit speed Jordan arc joining from $x$ to $x_n.$ 
It is clear that $c:=e_1$ and $e_n(n>1)$ are Jordan arcs emanating from the common cut  point $x.$
Suppose that $e_n(0,a_n]$ and $c(0,a],$ where $a:=a_1,$
 have no common point for some $n>1.$ 
Let $\{\epsilon_i\}$ be a decreasing sequence convergent to zero. Since $e_n(0,a_n]$ and $c(0,a]$ have no common point, we get the subarc $c_i$ (lying in $\Sigma_x$) of the circle centered at $x$ with radius $\epsilon_i$ cut off by $e_n$ and $c$ for each $i.$ Let $\gamma_i$ denote an $N$-segment to an interior point of $c_i$ for each $i.$ Then, any limit $N$-segment
of the sequence $\{\gamma_i\}$ as $i\to \infty,$ is an $N$-segment 
to $x$ lying in $\Sigma_x.$ 
This contradicts the definition of a sector. Therefore, there exists $t_n\in(0,a_n]$ satisfying $e_n=c$ on $[0,t_n]$ for each $n.$
From Theorem \ref{th3.6} and Lemma \ref{lem7.10},
$\lim_{n\to\infty} \delta(x,c(t_n))=0.$
Hence, by the triangle inequality, it is sufficient to prove 
$\lim_{n\to\infty}\delta (c(t_n),x_n)=0.$ It is obvious that each sector at $e_n(t)(t_n<t<a_n)$ containing $e_n(t,t+\delta)$ for small $\delta>0$ shrinks to an $N$-segment to $x$ as $n\to\infty.$ Therefore, by Lemma \ref{lem7.8}, there exists $\epsilon_0\in(0,1)$ such that
$$(d_N\circ e_n)'(t)<-\epsilon_0$$ 
for almost all $t\in(t_n,a_n)$ and for all $n.$
By integrating the equation above, we get
$\delta(c(t_n),x_n)<\frac{-1}{\epsilon_0} ( d(N,c(t_n))-d(N,x_n))$
for all $n.$ Since $\lim_{n\to\infty} d(N,c(t_n))=d(N,x)=\lim_{n\to\infty} d(N,x_n),$ we obtain $\lim_{n\to\infty} \delta(c(t_n),x_n)=0.$
$\qedd$
\end{proof}

\begin{lemma}\label{lem7.12}
Let   $\{x_n\}$  be a sequence of cut points of $N$ convergent to 
a cut point  $x.$
If there are no sectors at $x$  that contain an infinite subsequence of the sequence $\{x_n\},$ then
$\lim_{n\to\infty} \delta(x,x_n)=0.$
\end{lemma}
\begin{proof}
For each $n,$
let $\Sigma_n$ denote the sector at $x$ containing $x_n.$
Then, from the  hypothesis of our lemma, the sequence $\{\Sigma_n\}$
shrinks to an $N$-segment to $x.$
By applying the argument for the pair $x_n$ and $c(t_n)$ in 
the proof of Lemma \ref{lem7.11} to the pair $x$ and $x_n,$
we get a number $\epsilon_0\in(0,1)$ satisfying 
$\delta(x,x_n)<\frac{1}{\epsilon_0}\left(   d(N,x)-d(N,x_n) \right)$
for each $n.$
Hence we obtain
$\lim_{n\to\infty}\delta(x,x_n)=0.$
$\qedd$
\end{proof}

\begin{theorem}\label{th7.13}
Let $N$ be a closed subset of 
  a forward complete 2-dimensional Finsler manifold $(M,F)$ and 
  $C_N$  the cut locus of $N.$
  Then, the topology of $C_N$ induced from the intrinsic metric $\delta$ 
 coincides with the induced topology of $C_N$ from $(M,F).$

\end{theorem}
\begin{proof}
It is sufficient to prove that for any $x\in C_N$ and any sequence $\{x_n\}$ of cut points of $N,$
$\lim_{n\to\infty}\delta(x,x_n)=0$ if and only if $\lim_{n\to\infty} d(x,x_n)=0.$
Since $d(x,y)\leq \delta(x,y)$ for any $x,y\in C_N,$ it is trivial that 
$\lim_{n\to\infty}\delta(x,x_n)=0$ implies $\lim_{n\to\infty}d(x,x_n)=0.$
Suppose that $\lim_{n\to\infty} d(x,x_n)=0.$
By assuming that there exist an infinite subsequence $\{ x_{n_i} \}$ of $\{x_n\}$ and a positive constant $\eta$  satisfying  $\delta(x,x_{n_i})>\eta$ for any $n_i,$
we will get a contradiction.
We may assume that all $x_{n_i}$ lie in a common sector $\Sigma_x$ at $x$ or each $x_{n_i}$ is contained in a mutually distinct sector at $x,$ by choosing a 
subsequence of $\{x_{n_i}\}$ if necessary.
From Lemmas \ref{lem7.11} and \ref{lem7.12}, we get
$0<\eta\leq\lim_{n\to\infty}\delta(x,x_{n_i})=0.$ This is a contradiction.
$\qedd$
\end{proof}

\section{ Proof of the completeness with respect to the {intrinsic} metric  $\delta$}\label{sec8}

\noindent
Let $\{x_n\}$ denote a forward Cauchy sequence 
{of points in $C_{N}$}
with respect to $\delta.$
Here, without loss of generality, we may assume that $\delta(x_n,x_m)<\infty$ for all $n<m,$ i.e., all $x_n$ lie in a common connected component of $C_N.$
Since $d\leq \delta,$ the sequence is a forward Cauchy sequence with respect to $d.$ 
{The metric space} $(M,d)$ is  forward complete, {therefore} there exists a unique limit {point}
$\lim_{n\to\infty} x_n=:q.$ Since $\lim_{n\to\infty}d(x_n,q)=\lim_{n\to\infty}d(q,x_n)=0,$ we may choose a positive integer $n_1$ and the positive number $\delta_0$ chosen in Section 3 for the cut point $x:=x_{n_1}$ so as to satisfy
$q\in B_{\delta_0}(x).$
We fix the point $x=x_{n_1}.$ Choose any small positive number $\epsilon$ 
so as to satisfy 
\begin{equation}\label{eq:8-1}
d(q,x)>2\epsilon
\end{equation}
and
\begin{equation}\label{eq:8-2}
B_\epsilon(q)\subset B_{\delta_0}(x)
\end{equation}
{and fix it.} 
Since the sequence $\{x_n\}$ is a forward Cauchy sequence with respect to $\delta,$
we may choose a positive integer $n_0:=n_0(\epsilon)$ in such a way that
\begin{equation}\label{eq:8-3}
\delta(x_{n_0},x_{n_0+k})<\frac{\epsilon}{2}
\end{equation} 
for all $k>0$
and 
\begin{equation}\label{eq:8-4}
d(q,x_{n_0})<\frac{\epsilon}{2}.
\end{equation}
For each integer $k\geq 1,$
let $c_k:[0,a_k]\to C_N$ denote a unit speed Jordan arc joining $x_{n_0}$ to $x_{n_0+k}.$ 
By \eqref{eq:8-1},
we may assume that
\begin{equation}\label{eq:8-5}
a_k<\frac{\epsilon}2.
\end{equation}
\begin{lemma}\label{lem8.1}
For each $k\geq 1,$
$c_k[0,a_k]$ is a subset of $B_\epsilon(q).$
\end{lemma}
\begin{proof} 
It follows from Theorem \ref{th7.1}, Lemmas \ref{lem7.2} and \ref{lem7.6}
that $d(q,c_k(t))\leq d(q,x_{n_0})+t$  for any $t\in[0,a_k].$
Since $a_k<\frac{\epsilon}2 $
by \eqref{eq:8-5} and 
$d(q,x_{n_0})<\frac{\epsilon}2$ by \eqref{eq:8-4}, we obtain
$d(q,c_k(t))<\epsilon$ for any $t\in[0,a_k].$
\end{proof}
$\qedd$

\begin{lemma}\label{lem8.2}
There exists a sector $\Sigma_{q_\epsilon}$ at a cut point 
$q_\epsilon,$ which is not an endpoint of $C_N,$
  such that 
$q\in \Sigma_{q_\epsilon}$ and   $d(q,q_\epsilon)=2\epsilon.$
Hence  
there exists a sector $\Sigma_{q_{\epsilon_1}}$ at a cut  point $q_{\epsilon_1}$ of $N,$ which is not an endpoint of $C_N,$ such that 
$q\in\Sigma_{q_{\epsilon_1}}\subset \Sigma_{q_\epsilon }$ 
and  $d(q,q_{\epsilon_1})=2\epsilon_1$
for some $0<\epsilon_1<\epsilon.$ 
\end{lemma}
\begin{proof}
Let $c:[0,b]\to C_N$ denote a unit speed Jordan arc joining $x$ to $x_{n_0}.$ Let $c(t_0)$ denote a point on  the arc $c$ with $d(q,c(t_0))=\epsilon.$
The existence of $c(t_0)$ is clear, since  $d(q,c(0))=d(q,x)>2\epsilon$ and $d(q,c(b))=d(q,x_{n_0})<\frac{\epsilon}{2}$ by \eqref{eq:8-1} and \eqref{eq:8-4}.
Let $\Sigma_{c(t_0)}^+$ denote the sector at $c(t_0)$ containing $x_{n_0}.$
By Lemma \ref{lem8.1},
$x_{n_0+k}\in\Sigma_{c(t_0)}^+$ for all $k\geq 1.$ Hence,
the point $q$ is an element of the closure of $\Sigma_{c(t_0)}^+.$ 
Since the closure of $\Sigma_{c(t_0)}^+$ is a  subset of $\Sigma_{c(t)}^+$ for any $t<t_0,$
$q$ is an element of  $\Sigma_{q_\epsilon}^+,$
where  $\Sigma_{c(t)}^+$ denotes the sector at $c(t)$ containing $x_{n_0}$
 and $q_\epsilon$ denotes a point on the arc $c$ with $d(q,q_\epsilon)=2\epsilon.$
Hence the sector $\Sigma_{q_\epsilon}^+$ has the required property.

\end{proof}
$\qedd$

\begin{lemma}\label{lem8.3}

Let $\{\Sigma_{n}\}$ be a decreasing sequence of sectors (i.e., $\Sigma_1\supset\Sigma_2\supset\Sigma_3\supset\dots$ ) 
{such that each $\Sigma_{n}$ is a sector} 
at a cut point $q_n$ of $N.$
Suppose that $\lim_{n\to\infty} q_n:=q$ exists and $q\in\Sigma{_n}$ for all $n.$
If $q_n$ is not an endpoint of $C_N$ for all $n,$ then $q$ is a cut point of $N.$

\end{lemma}
\begin{proof}
For each $n,$ let $\alpha_n$ and $\beta_n$ denote the $N$-segments that form part of the boundary of $\Sigma_{n}.$ Since $q_n$
 is not an endpoint of $C_N
$ for each $n,$
$\alpha_n\ne\beta_n$ for each $n.$
Suppose first that there exists a
subsequence of $\{\Sigma_{n}\}$
which does not shrink to a single $N$-segment. Then, there exist at least  two $N$-segments
to $q.$
This implies that $q$ is a cut point of $N.$ Suppose next that the sequence shrinks to a single 
$N$-segment.
Then, $\{\alpha_n\}$ and $\{\beta_n\}$ shrink to a common $N$-segment $\gamma :[0,l]\to M$
to $q=\gamma(l).$
Let $\tilde \gamma :[0,\infty)\to M$ denote the geodesic extension of $\gamma.$ For any sufficiently large $n,$ $\tilde\gamma$ intersect the $N$-segment $\alpha_n$ or $\beta_n$ at a point $\tilde \gamma(l_n), $ $l_n>l.$ 
Since $\lim_{n\to\infty} \tilde\gamma(l_n)=\gamma(l)=q,$ and 
 $\tilde\gamma|_{[0,l_n]}$ is not an $N$-segment,  $q$ is a cut point of $N.$

\end{proof}
$\qedd$

\begin{theorem}\label{th8.4}
Let $N$ be  a closed subset of  a forward complete 2-dimensional Finsler manifold $(M,F).$
Then the cut locus $C_N$ of $N$ with the {intrinsic} distance $\delta$ is forward complete.

\end{theorem}

\begin{proof}
Let $\{x_n\}$ denote a forward Cauchy sequence with respect to $\delta.$ Then, there exists a unique limit $\lim_{n\to \infty} x_n=:q$ with respect to $d,$ since $d\leq \delta$ and $(M,d)$ is forward complete. By Lemma \ref{lem8.2}, there exists a decreasing sequence of sectors $\Sigma_{n}$ at cut points $q_n$ such that $\lim_{n\to\infty} q_n=q$ with respect to $d$ and none of $q_n$ is an endpoint of $C_N.$ Hence, by Lemma \ref{lem8.3}, $q$ is  a cut point of $N.$ From Theorem \ref{th7.13}, we obtain $\lim_{n\to\infty}\delta(x_n,q)=0.$

\end{proof}
$\qedd$

\section{The proof of Theorem C}\label{sec9}

For each sector $\Sigma$ at a cut point $x$ of $N,$ 
we define a number $\mu(\Sigma)$ by
$$\mu(\Sigma):=\max\{g_X(X,Y),g_Y(Y,X)\},$$
where $X$ and $Y$ denote the velocity vectors at $x$ of the two unit speed
$N$-segments that form part of the boundary of $\Sigma.$
It follows from Lemma 1.2.3  in \cite{S} that $\mu(\Sigma)<1$ if $X\ne Y.$
{Recall that} if a cut point $x$ of $N$ admits more than two sectors, then $x$ is called a branch cut point {(see Remark \ref{rem7.9})}.

For each $n=1,2,3,...,$ let $A_n$ denote the subset of $C_N$ which consists of all branch cut points that admit three sectors $\Sigma^i,i=1,2,3,$ satisfying $\mu(\Sigma^i)\leq 1-\frac{1}{n}.$ 
It is clear that $\bigcup_{n=1}^\infty A_n$ is the set of all branch cut points of $C_N.$

 \begin{lemma}\label{lem9.1}
For each $n,$  the set $A_n$ is locally finite.  Hence, the set of all branch cut points is at most countable.
\end{lemma}
\begin{proof}
By assuming that there exists a ball $B_r(x) ,(0<r<\infty)$
containing infinitely many elements $z_\alpha$ of $A_n$ for some $n,$
we will get a contradiction.
The set $\{z_\alpha\}$ has an accumulation point $z.$ The  point $z$ is also an element
of $A_n,$ since any two $N$-segments forming the three sectors $\Sigma$ with $\mu(\Sigma)\leq 1-\frac{1}{n}<1$ cannot shrink to a single $N$-segment.
 Let $\{z_j\}$ denote a sequence of points of $\{z_\alpha\}$ convergent to $z.$ Without loss of generality, we may assume that all $z_j$ are in a common sector at $z$ or each $z_j$ lies in a mutually distinct sector at $z,$ by taking a subsequence if necessary.
 Thus, there exist at most two limit $N$-segments to $z$ of the sequence of  $N$-segments to $z_j.$ 
  This contradicts the fact $z_j\in A_n.$

$\qedd$
\end{proof}
By Lemma \ref{lem9.1}, there exist at most countably many branch cut points, but we do not know if the closure  $A_N$ of  $\bigcup_{n=1}^\infty A_n$  is countable or not.
Here, we choose a tree $T\subset C_N\cap B_{\delta_0}(x),$ where $x$ is a cut point of $N$ and $\delta_0$ is the positive number chosen in Section \ref{sec3}.   
We define a subset $T^b$ of $T$ by
$$T^b:=\{ y\in A_N \cap T\:| \: y \textrm{ admits a sector having no branch cut points in } T\}.$$
\begin{lemma}\label{lem9.2}
The set  $T^b$ is countable.
\end{lemma}
\begin{proof}
For each element $y\in T^b,$
there exists the subarc $c_y$ of $S_{\delta_0}(x)$ cut off by the sector at $y$ that has no branch cut points. It is clear that $c_{y_1}\cap c_{y_2}=\emptyset$  if $y_1\ne y_2.$ Since there exist at most countably many non-overlapping subarcs of $S_{\delta_0}(x),$ {it follows that} $T^b$ is countable. 
$\qedd$
\end{proof}

\begin{theorem}\label{th9.3}
The set $C_N\setminus C_N^{\ e}$ is a union of countably many Jordan arcs, 
 where $C_N^{\ e}$ denotes the set of all endpoints of $C_N$.
\end{theorem}
\begin{proof}
By Lemmas \ref{lem9.1} and \ref{lem9.2}, there exist at most countably many elements
$\{x_i\:| \: i=1,2,3,...\}$ in $T^b\cup(T\cap\bigcup_{n=1}^\infty A_n).$
 Since  it is trivial {to see} that $T$ consists of a unique Jordan arc if  $T$
  has no branch  cut points, we may assume that $x_1$ is a branch cut point. For each $x_i(i>1),$ let $m_i:[0,a_i]\to T$ be the unique Jordan arc joining from  $x_1$ to  $x_i.$ Choose any $q\in T\setminus \bigcup_{i=2}^\infty |m_i|,$ where $|m_i|:=m_i[0,a_i].$ 
Let $c:[0,b]\to T$ be the Jordan arc joining from $x_1$ to $q.$ If $q$ is not an endpoint of the cut locus $C_N, $ $c$ has an extension $\tilde c\: : \:
[0,\tilde b]\to T.$ Then, $q\notin \bigcup_{i=2}^\infty |m_i|$ implies that $\tilde c|_{(b,\tilde b)}$ does not intersect any $x_i$ and hence  has no branch cut points.
Let $b_1(<b)$ be the maximum number $b_1$ with $c(b_1)=x_j$ for some $j.$
Then, $q$ lies in a Jordan arc (without any branch cut points except $x_j$) emanating from some $x_j.$ At each $x_i$ there exist at most countably many such Jordan arcs in $T.$ Therefore,
$T\setminus(C_N^{\ e}\cup\bigcup_{i=2}^\infty|m_i|)  $ is a union of countably many Jordan arcs.
This implies that $C_N\setminus C_N^{\ e} $ is a union of countably many Jordan arcs.

$\qedd$
\end{proof}

\begin{remark}\label{rem9.4}
Recall that even in the Riemannian case, there are compact convex surfaces of revolution such that the cut locus of a point on the surface admits a branch cut points with infinitely many ramifying branches (\cite{GS}).
\end{remark}

 A critical point of  the distance function on a Finsler manifold  is defined analogously to the Riemannian distance function (see \cite{C}), i.e.,  
a point $q\in M\setminus N$ is called a {\it  critical} point of the distance  function $d_N$ from $N$ if for any 
 tangent vector  $v$ at $q,$
there exists an  $N$-segment $\gamma:[0,l]\to M$ to $q=\gamma(l)$ such that $g_{\dot\gamma(l)}(\dot\gamma(l),v)\leq 0.$ It is trivial that any critical point of $d_N$ admits at least two $N$-segments, and hence any critical point is
 a cut point of $N.$
Notice that the Gromov isotopy lemma (\cite{C}) holds for the distance function $d_N.$ The proof of the isotopy lemma for the distance function $d_N$ is the same as the Riemannian case. 

\begin{lemma}\label{lem9.5}
Let $c:[a,b]\to C_N$ be a unit speed Jordan arc. Suppose that $c(t)$ and $(d_N\circ c)(t)$ are  differentiable at $t=t_0\in(a,b).$  If $c(t_0)$ is a critical point of $d_N,$ then
$(d_N\circ c)'(t_0)=0.$   

\end{lemma}
\begin{proof}
By supposing that $(d_N\circ c)
'(t_0)\ne 0,$ we will get a contradiction.
We may assume that $(d_N\circ c)'(t_0)>0,$ by reversing the parameter of $c$ if necessary. From Proposition \ref{prop2.1}, it follows that
$$0<(d_N\circ c)'(t_0)=g_X(X,\dot c(t_0))\leq g_Y(Y,\dot c(t_0))$$
for the velocity vector $Y$ at $c(t_0)$ of any $N$-segment to $c(t_0).$
Here $X$ denotes the velocity vector at $c(t_0)$ of an $N$-segment that form part of the boundary of the sector $\Sigma^+_{c(t_0)}$ at $c(t_0).$
Hence, $g_Y(Y,\dot c(t_0))>0$ for the velocity vector $Y$ at $c(t_0)$ of any
 $N$-segment to $c(t_0).$  This contradicts the fact that $c(t_0)$ is a 
critical point of $d_N.$

$\qedd$
\end{proof}

\begin{lemma}
\label{lem9.6}
For each unit speed Jordan arc $c:[a,b]\to C_N,$ there exists
 a measure zero subset $\cal E$ of $d_N\circ c[a,b]$ such that
  if $(d_N\circ c)(t)\notin {\cal E},$ then 
$(d_N\circ c)'(t)\ne 0.$
\end{lemma}
\begin{proof}
It was proved in Lemma 3.2 of \cite{ShT} that
 the Sard theorem holds for a continuous function of one variable, i.e.,
 the set 
 $${\cal E}_1:=\{d_N\circ c(t) \; | \: (d_N\circ c)\:  \textrm{is differentiable at}\:  t\in(a,b)\: \textrm{and} \: (d_N\circ c)'(t)=0\}$$
 is  of measure zero.
 On the other hand, $d_N\circ c$ is differentiable almost everywhere, since
  $d_N\circ c$ is a Lipschitz function. Hence, the image ${\cal E}_2$ of non-differentiable points of $d_N\circ c$ by this Lipschitz function is of measure zero. Therefore, 
 the set ${\cal E}:={\cal E}_1\cup{\cal E}_2\cup\{d_N(c(a)),d_N(c(b))\}$ is 
of measure zero and satisfies the required properties.
$\qedd$
\end{proof}

\begin{theorem}\label{th9.7}
Let $N$ be a closed subset of a 2-dimensional Finsler manifold $(M,F).$
Then, there exists a subset ${\cal E}\subset[0,\sup d_N)$ of measure zero such that for any $t\in(0,\sup d_N)\setminus {\cal E},$ {the set} $d_N^{-1}(t)$ is a union of disjoint continuous curves. Furthermore, any point $q\in d_N^{-1}(t)$ admits at most 
two $N$-segments. In particular, $d_N^{-1}(t)$ is a union of finitely many 
disjoint circles if $N$ is compact.
\end{theorem}
\begin{proof}
Let $C_N^b$ denote the set consisting  of all branch cut points of $N.$ Since 
$C_N^b$ is at most countable by Lemma \ref{lem9.1},
the set ${\cal E}^b:= d_N(C_N^b)$ is of measure zero. 
By Theorem \ref{th9.3},
$C_N\setminus C_N{}^e=\bigcup_{i=i}^\infty |c_i|,$
where $c_i: [a_i,b_i]\to C_N$ denotes a unit speed Jordan arc.
Hence, by applying Lemmas \ref{lem9.5} and \ref{lem9.6} for each $c_i,$ there exists a measure zero set ${\cal E}_i\subset (d_N\circ c_i)[a_i,b_i]$  such that
 if $(d_N\circ c_i)(t)\notin {\cal E}_i,$ then 
$(d_N\circ c_i)'(t)\ne 0.$
 Let $C_N{}^c$ denote the set consisting of all endpoints of $C_N $ admitting  more than one $N$-segment.
Since $C_N{}^c$ is a countable set, the set ${\cal E}^c:=d_N(C_N{}^c)$ is of measure zero.
Thus,
the set 
${\cal E}:= \bigcup_{i=1}^\infty {\cal E}_i \cup {\cal E}^b\cup {\cal E}^c$ is of measure zero. 
Choose any $s\in(0,\sup d_N)\setminus{\cal E}, $ and fix it.
Suppose that $d_N^{-1}(s)$ contains a critical point $q$ of $d_N.$
If the point $q$ is an endpoint of $C_N,$ then the point is an endpoint admitting more than one $N$-segment. Hence $s=d_N(q)\in {\cal E}^c,$ which contradicts our assumption $s\notin {\cal E}.$  
This implies that $q\in C_N\setminus C_N{}^e=\bigcup_{i=1}^\infty |c_i|$ and 
the point is a critical point of  $d_N$ on some $c_i[a_i,b_i].$  Suppose that $q=c_i(t_i)$ for some $t_i\in[a_i,b_i].$
Since $s\notin {\cal E}_i,$
$(d_N\circ c_i)(t_i)\ne 0.$ This is a contradiction by Lemma \ref{lem9.5}, since $q=c_i(t_i)$ is a critical point of $d_N.$ Therefore, any point $q\in M\setminus N$ with $d_N(q)\in(0,\sup d_N)\setminus{\cal E}$ is not a critical point and admits at most two $N$-segments.

$\qedd$
\end{proof}

\bigskip

\medskip

\begin{center}
Minoru TANAKA 

\bigskip
Department of Mathematics\\
Tokai University\\
Hiratsuka City, Kanagawa Pref.\\ 
259\,--\,1292 Japan

\medskip
{\tt tanaka@tokai-u.jp}

\bigskip

Sorin V. SABAU\\

\bigskip

Department of Mathematics\\
Tokai University\\
Sapporo City, Hokkaido\\
005\,--\,8601 Japan

\medskip
{\tt sorin@tspirit.tokai-u.jp}
\end{center}

\end{document}